\documentclass[12pt]{article}
\usepackage[french, english]{babel}
\usepackage{graphicx}
\usepackage{amsmath}
\usepackage{amsfonts}
\usepackage{amssymb}
\usepackage{amsthm}
\usepackage{hyperref}
\usepackage{mathtools}
\usepackage{comment}

\input xy
\xyoption{all}

\numberwithin{equation}{section}

\newtheorem{thm}{Th\'eor\`eme}
\newtheorem{cor}[thm]{Corollaire}
\newtheorem{lem}[thm]{Lemme}
\newtheorem{pro}[thm]{Proposition}

\theoremstyle{definition}
\newtheorem{defi}[thm]{D\'efinition}

\newtheoremstyle{remarque}{}{}{}{}{\it}{.}{\newline}{}
\theoremstyle{remarque}
\newtheorem*{rem}{Remarque}

\newcommand{\asd}[5]{% 
\setbox1=\hbox{\ensuremath{^{#1}}}% 
\setbox2=\hbox{\ensuremath{_{#2}}}% 
\setbox5=\hbox{\ensuremath{#5}}% 
\hspace{\ifnum\wd1>\wd2\wd1\else\wd2\fi}% 
\ensuremath{\copy5^{\hspace{-\wd1}\hspace{-\wd5}#1\hspace{\wd5}#3}% 
_{\hspace{-\wd2}\hspace{-\wd5}#2\hspace{\wd5}#4}% 
}}

\usepackage[OT2,T1]{fontenc}
\DeclareSymbolFont{cyrletters}{OT2}{wncyr}{m}{n}
\DeclareMathSymbol{\Sha}{\mathalpha}{cyrletters}{"58}
\DeclareMathSymbol{\Brusse}{\mathalpha}{cyrletters}{"42}

\newcommand{\z}{\mathbb{Z}}
\newcommand{\q}{\mathbb{Q}}
\newcommand{\f}{\mathbb{F}}

\newcommand{\im}{\mathrm{Im}}

\newcommand{\res}{\mathrm{Res}}
\newcommand{\cores}{\mathrm{Cor}}
\newcommand{\aut}{\mathrm{Aut}\,}
\newcommand{\saut}{\mathrm{SAut}\,}

\renewcommand{\int}{\mathrm{Int}\,}
\newcommand{\gm}{\mathbb{G}_{\mathrm{m}}}
\newcommand{\br}{\mathrm{Br}\,}
\newcommand{\nr}{\mathrm{nr}}

\newcommand{\brun}{\mathrm{Br}_1}
\newcommand{\brnr}{{\mathrm{Br}_{\nr}}}
\newcommand{\al}{\mathrm{al}}
\newcommand{\bral}{\mathrm{Br}_{\al}}
\newcommand{\brnral}{\mathrm{Br}_{\nr,\al}}
\newcommand{\gal}{\mathrm{Gal}}
\newcommand{\sln}{\mathrm{SL}_n}

\newcommand{\pic}{\mathrm{Pic}\,}

\newcommand{\ab}{\mathrm{ab}}
\newcommand{\tor}{\mathrm{tor}}

\newcommand{\red}{\mathrm{red}}
\newcommand{\spec}{\mathrm{Spec}\,}
\renewcommand{\cal}[1]{\mathcal{#1}}
\newcommand{\bb}[1]{\mathbb{#1}}
\newcommand{\torf}{{\rm torf}}

\renewcommand{\b}{\mathfrak{b}}

\title{Groupe de Brauer non ramifi\'e alg\'ebrique des espaces homog\`enes}
\author{Giancarlo Lucchini Arteche\\[5mm]
{\it\small D\'epartement de math\'ematiques, universit\'e Paris-Sud}\\
{\it\small b\^atiment 425, 91405 Orsay cedex, France}\\
{\small giancarlo.lucchini@math.u-psud.fr}
}

\date{}

\begin{document}

\selectlanguage{french}
\maketitle

\begin{abstract}
Via une r\'eduction de la cohomologie galoisienne d'un groupe alg\'ebrique lin\'eaire $G$ \`a celle d'un certain sous-quotient fini, on donne diff\'erentes
formules permettant de calculer le groupe de Brauer non ramifi\'e alg\'ebrique d'un espace homog\`ene $V=G\backslash G'$ avec $G'$ semi-simple simplement connexe
tant sur un corps fini que sur un corps de caract\'eristique 0.\\

Mots cl\'es : groupe de Brauer, espaces homog\`enes, cohomologie galoisienne.
\end{abstract}

\selectlanguage{english}
\begin{abstract}
{\bf The unramified algebraic Brauer group of homogeneous spaces.} Using a reduction of the Galois cohomology of a linear algebraic group $G$ to that of a certain
finite subquotient, we give different formulas allowing the calculation of the unramified algebraic Brauer group of a homogeneous space $V=G\backslash G'$ with
$G'$ semisimple and simply connected, both over a finite field and over an arbitrary field of characteristic $0$.\\

\end{abstract}

\selectlanguage{french}

\section*{Introduction}
Dans \cite{GLABrnral}, on a donn\'e plusieurs formules permettant de calculer le groupe de Brauer non ramifi\'e alg\'ebrique $\brnral V$ d'un espace homog\`ene
$V=G\backslash G'$ sous un $k$-groupe alg\'ebrique $G'$ semi-simple simplement connexe \`a stabilisateur fini $G$ pour $k$ un corps fini ou bien quelconque de
caract\'eristique $0$. Ces travaux \'etaient un compl\'ement naturel \`a ceux de Borovoi, Demarche et Harari dans \cite{BDH}, texte qui donne des formules
pour le m\^eme pour des espaces homog\`enes (avec des groupes ambiants plus g\'en\'eraux) \`a stabilisateur de type ``ssumult'', cas qui contient notamment
celui des groupes connexes et des groupes commutatifs. Or, une fois trait\'e le cas des stabilisateurs finis, qui sont parmi les groupes non connexes et non
commutatifs les plus simples que l'on puisse imaginer, la question naturelle \`a se poser est dans quelle mesure on peut \'etendre ces formules au cas d'un
stabilisateur \emph{quelconque}, \'etant donn\'e que tout groupe alg\'ebrique lin\'eaire est une extension d'un groupe fini par un groupe connexe. L'objectif de ce
texte est pr\'ecis\'ement la r\'ealisation de cette d\'emarche.

Pour arriver \`a un tel but, on s'est inspir\'e d'un contre-exemple aux formules de Borovoi, Demarche et Harari, fourni par eux-m\^emes dans
\cite[Prop. 8.4]{BDH}. Ce r\'esultat montre que l'on peut obtenir des groupes qui ne sont pas de type ``ssumult'' pour lesquels la formule qu'ils ont
d\'evelopp\'ee ne marche plus. Or, cet exemple consiste tout simplement en la donn\'ee d'un groupe fini non ab\'elien ne v\'erifiant pas leur formule
(et dans \cite{GLABrnral} on a bien montr\'e qu'il en est ainsi pour certains groupes finis non ab\'eliens), puis d'un sous-groupe ab\'elien de celui-ci
que l'on plonge dans un tore, obtenant ainsi une extension non triviale d'un groupe fini par un tore. Le fait que le groupe de Brauer non ramifi\'e alg\'ebrique
d'un espace homog\`ene avec un tel stabilisateur soit le m\^eme que celui d'un espace homog\`ene avec le stabilisateur fini dont ils sont partis sugg\`ere
le principe technique suivant :\\

\emph{Pour un espace homog\`ene $V=G\backslash G'$ \`a stabilisateur $G$ non connexe, toute l'information concernant le groupe de Brauer non ramifi\'e alg\'ebrique
$\brnral V$ est contenue dans un sous-groupe fini de $G$, extension du groupe des composantes connexes de $G$ par un sous-groupe fini de $G^\circ$, la composante
connexe neutre de $G$.}\\

\noindent Ainsi, on est d'abord amen\'e \`a r\'esoudre :
\begin{enumerate}
\item la question de l'\emph{existence} de tels sous-groupes ;
\item la question de savoir si effectivement ils ``contiennent toute l'information''.
\end{enumerate}
\vspace{3mm}

La premi\`ere question, laquelle avait d\'ej\`a \'et\'e pos\'ee par Chernousov, Gille et Reichstein dans
\cite[Rem. 3.8]{Gille_Ch_R}, est trait\'ee d'une fa\c con satisfaisante dans ce texte. On montre en effet le r\'esultat suivant :

\begin{pro}\label{proposition existence intro}{\rm [Proposition \ref{proposition existence du sous-groupe fini}]}\\
Soit $k$ un corps de caract\'eristique $p\geq 0$. Soit $G$ un $k$-groupe alg\'ebrique lin\'eaire, extension d'un groupe fini $F$ d'ordre $n$ premier \`a $p$ par un tore $T$
d\'eploy\'e par une extension $L/k$ de degr\'e $d$, aussi premier \`a $p$. Alors il existe un $k$-sous-groupe fini $H$ de $G$ et un diagramme commutatif
\`a lignes exactes
\[\xymatrix{
1 \ar[r] & \asd{}{nd}{}{}{T} \ar[r] \ar@{^{(}->}[d] & H \ar[r] \ar@{^{(}->}[d] & F \ar[r] \ar@{=}[d] & 1 \\
1 \ar[r] & T \ar[r] & G \ar[r] & F \ar[r] & 1,
}\]
o\`u $\asd{}{nd}{}{}{T}$ d\'esigne le sous-groupe de $nd$-torsion de $T$.
\end{pro}

Si l'on r\'epond \`a la question seulement pour les extensions d'un groupe fini par un tore, il faut remarquer que cela suffit pour notre objectif car, de m\^eme
que chez \cite{BDH}, la partie unipotente et la partie semi-simple du stabilisateur ne jouent aucun r\^ole dans les formules (d'o\`u le fait de se concentrer sur
les groupes de type multiplicatif chez \cite{BDH}). Ainsi, ce sera le quotient $G^\torf$ (cf. les notations ci-dessous) et notamment un sous-groupe fini de celui-ci
(donc un sous-quotient fini de $G$) qui portera toute l'information et sera donc utilis\'e pour calculer le groupe $\brnral V$.\\

La r\'eponse \`a la deuxi\`eme question est donn\'ee par un corollaire \'evident de la proposition 3.1 de \cite{Gille_Ch_R},
que l'on peut \'enoncer comme suit :

\begin{pro}\label{proposition Gille_Ch_R}
Soit $k$ un corps de caract\'eristique $p\geq 0$. Soit $G$ un $k$-groupe alg\'ebrique lin\'eaire, extension d'un groupe fini $F$ d'ordre $n$ premier \`a $p$ par un tore $T$
d\'eploy\'e par une extension $L/k$ de degr\'e $d$, aussi premier \`a $p$. Soit enfin $m=nd$. On suppose qu'il existe un sous-groupe fini $H_0$ de $G$ tel que
la compos\'ee $H_0\to G\to F$ est surjective et l'on note $H$ le $k$-sous-groupe fini de $G$ engendr\'e par $H_0$ et par $\phi_m^{-1}(H_0\cap T)$, o\`u $\phi_m$
repr\'esente la multiplication par $m$ dans $T$.

Alors, pour tout corps $k'\supset k$, l'application naturelle $H^1(k',H)\to H^1(k',G)$ est surjective.
\end{pro}

La surjectivit\'e de l'application $H_0\to F$ est bien entendu au niveau des $\bar k$-points. Autrement dit, on peut voir $F$ comme un quotient de
$H_0$. On remarque par ailleurs qu'il est facile de v\'erifier que le sous-groupe $H$ ci-dessus est bien un groupe fini. En effet, pour $h_0\in H_0$ et
$h_1\in H_1:=\phi_m^{-1}(H_0\cap T)$, on a $h_0h_1h_0^{-1}\in H_1$, ce qui montre que l'on peut \'ecrire tout \'el\'ement $h$ de $H$ comme un
produit $h=h_0h_1$ avec $h_0\in H_0$ et $h_1\in H_1$. Ces deux sous-groupes \'etant finis, la finitude de $H$ en d\'ecoule.

\begin{rem}
Cette proposition, ou encore la proposition 3.1 de \cite{Gille_Ch_R}, est un avatar (et une partie importante de la d\'emonstration) du th\'eor\`eme 1.2 du
m\^eme article. Ce dernier th\'eor\`eme vise notamment \`a g\'en\'eraliser un autre r\'esultat des m\^emes auteurs (cf. \cite[Thm. 1.1]{Gille_Ch_R_2}),
lequel g\'en\'eralise pour sa part un r\'esultat de Galitski\u\i\, assurant l'existence des \emph{quasi-sections} (cf. \cite{Galitskii} ou
\cite[Thm. 4.2]{Gille_Ch_R_2}). En effet, les r\'esultats de Chernousov, Gille et Reichstein peuvent \^etre intepret\'es dans le langage des
\emph{$(G,H)$-sections} (cf. \cite[Thm. 1.1']{Gille_Ch_R_2}), notion due \`a Katsylo (cf. \cite{Katsylo}). Sous cette traduction, on peut prouver facilement
le r\'esultat de Galitski\u\i, ainsi que le ``lemme sans nom'' (cf. \cite[Lem. 4.4]{Gille_Ch_R_2} ou encore \cite[Cor. 3.9]{ColliotSansucChili}), dont la
version la plus ancienne semble se trouver dans \cite{Bogomolov_Katsylo}. Pour un survol de ces notions, on renvoie \`a \cite{Popov}.
\end{rem}

Lorsqu'on associe cette proposition de nature cohomologique \`a la formule cohomologique dans \cite[Thm. 4.1]{GLABrnral}, on est capable de suivre
la d\'emarche faite dans \cite{GLABrnral} pour obtenir, dans l'ordre,
\begin{itemize}
\item une formule cohomologique pour $\brnral V$ sur un corps fini (Th\'eor\`eme \ref{theoreme cohomologique}) ;
\item une formule alg\'ebrique pour $\brnral V$ sur un corps fini (Th\'eor\`eme \ref{theoreme formule sur fq}) ;
\item une formule alg\'ebrique pour $\brnral V$ sur un corps de caract\'eristique $0$ (Th\'eor\`eme \ref{theoreme formule en car 0}) ;
\item une formule cohomologique pour $\brnral V$ sur un corps local de caract\'eristique $0$, sous quelques hypoth\`eses suppl\'ementaires (Proposition
\ref{proposition cohomologique corps local}).
\end{itemize}
Ici, de m\^eme que dans \cite{GLABrnral}, on entend par formule cohomologique (resp. alg\'ebrique) une formule d\'ependant des th\'eor\`emes de
dualit\'e en cohomologie galoisienne pour des corps locaux (resp. une formule d\'ependant de la structure de groupe alg\'ebrique du stabilisateur).\\

La composition de cet article est donc la suivante. Dans la section \ref{section notations et rappels}, on fixe toutes les notations et l'on donne quelques rappels
pr\'eliminaires, tous d\'ej\`a mentionn\'es dans \cite{GLABrnral}. Les r\'esultats sur un corps fini, modulo
l'existence d'un ``bon'' sous-quotient fini, sont donn\'es dans la section \ref{section formules corps fini}. La section \ref{section formules caracteristique 0}
fait de m\^eme pour les corps de caract\'eristique $0$. On a d\'ecid\'e de pr\'esenter ces r\'esultats avant ceux d'existence parce que, dans beaucoup
d'exemples de groupes lin\'eaires non connexes, il est possible de trouver des sous-groupes finis explicites avec lesquels on peut faire les calculs,
rendant ainsi inutile la proposition \ref{proposition existence intro}. Dans la section \ref{section existence} on d\'emontre la proposition
\ref{proposition existence intro} et l'on \'enonce \`a nouveau les r\'esultats pr\'ec\'edents (Corollaires \ref{corollaire formule corps finis avec existence} et
\ref{corollaire formule car 0 avec existence}) dans le cas o\`u l'on n'aurait pas acc\`es imm\'ediat \`a des ``bons'' sous-quotients finis. On profite aussi pour
donner une g\'en\'eralisation (corollaire \ref{corollaire sous-groupe fini surjectif en cohomologie 2}) d'un des r\'esultats de Chernousov, Gille et Reichstein,
\`a savoir leur corollaire 1.4 dans \cite{Gille_Ch_R}, que l'on rend valable sur un corps quelconque (alors que leur r\'esultat ne vaut que sur un corps alg\'ebriquement
clos). Enfin, dans la section \ref{section BM}, on donne un exemple d'application de ces r\'esultats. Il s'agit de l'obstruction de Brauer-Manin associ\'ee au groupe
$\brnral V$, pour laquelle on montre qu'elle ``ne voit pas les places r\'eelles'', tout comme dans \cite[\S 5.2]{GLABrnral} dans le cas des stabilisateurs finis.

Ce texte a aussi une section en appendice o\`u l'on pr\'esente une petite am\'elioration d'une remarque non publi\'ee faite par Colliot-Th\'el\`ene (cf. 
\cite{ColliotBrnr}) concernant le groupe de Brauer non ramifi\'e des espaces homog\`enes sous un groupe semi-simple simplement connexe. Le r\'esultat en question
est notamment utilis\'e dans la preuve du th\'eor\`eme \ref{theoreme formule en car 0} pour se ramener au cas o\`u le stabilisateur a une composante connexe neutre
r\'eductive.

\paragraph{Remerciements}
L'auteur tient \`a remercier David Harari et les rapporteurs pour la lecture soigneuse qu'ils ont fait de ce texte, ainsi que l'\'editeur pour ses
commentaires. Jean-Louis Colliot-Th\'el\`ene, Cyril Demarche et Philippe Gille m\'eritent par ailleurs toute ma gratitude pour de tr\`es importantes
discussions \`a l'origine de ces r\'esultats.

\section{Notations et rappels pr\'eliminaires}\label{section notations et rappels}
Dans tout ce texte, $k,k'$ et $L$ repr\'esentent des corps de caract\'eristique $p\geq 0$. Pour un corps $k$, on note toujours $\bar k$ une cl\^oture s\'eparable de
$k$ et $\Gamma_k:=\gal(\bar k/k)$ le groupe de Galois absolu. Dans le cas des corps finis, $q$ repr\'esente toujours le cardinal du corps
et il s'agit donc d'une puissance de $p$.

On note $V$ une $k$-vari\'et\'e, laquelle est toujours suppos\'ee lisse et g\'eom\'etriquement int\`egre. On notera $X$ une
$k$-compactification de $V$ (aussi suppos\'ee lisse), i.e. une $k$-vari\'et\'e propre munie d'une $k$-immersion ouverte $V\hookrightarrow X$.
\footnote{On rappelle que l'existence d'une compactification lisse est assur\'ee en caract\'eristique
0 via le th\'eor\`eme d'Hironaka, tandis qu'en caract\'eristique positive la dite existence est une question toujours ouverte.} Soit $G$ un
$k$-groupe alg\'ebrique lin\'eaire dont la composante connexe neutre est toujours suppos\'ee r\'eductive lorsque $k$ est de caract\'eristique positive. 
Pour un tel groupe, on note
\begin{itemize}
\item $D(G)$ le sous-groupe d\'eriv\'e de $G$ ;
\item $G^\ab=G/D(G)$ l'ab\'elianis\'e de $G$ ;
\item $\hat G=\hat G^\ab$ le $\Gamma_k$-module des caract\`eres de $G$ ;
\item $G^\circ$ la composante connexe de l'\'el\'ement neutre de $G$ ;
\item $F=G/G^\circ$ le groupe des composantes connexes de $G$ (c'est un groupe fini) ;
\item $R_\mathrm{u}(G)$ le radical unipotent de $G^\circ$ (c'est un groupe unipotent, trivial en caract\'eristique positive d'apr\`es notre hypoth\`ese) ;
\item $G^\red=G^\circ/R_\mathrm{u}(G)$ (c'est un groupe r\'eductif) ;
\item $G^\tor$ l'ab\'elianis\'e de $G^\red$ (c'est un tore) ;
\item $G^\torf=G/\ker[G^\circ\twoheadrightarrow G^\tor]$ (c'est une extension de $F$ par $G^\tor$).
\end{itemize}
On remarque que le sous-groupe $R_\mathrm{u}(G)$ est caract\'eristique dans $G$, ce qui nous dit que $\ker[G^\circ\twoheadrightarrow G^\tor]$ est
bien distingu\'e dans $G$ et donc que le quotient $G^\torf$ est bien d\'efini. Enfin, on supposera partout que l'ordre du groupe $F$ (\`a savoir, le
degr\'e du morphisme structural $F\to \spec k$) est premier \`a la caract\'eristique $p$ du corps $k$. Le cardinal de $F(\bar k)$ est ainsi suppos\'e
toujours premier \`a $p$. On regardera souvent le $k$-groupe $G$ comme immerg\'e dans un $k$-groupe semi-simple simplement connexe (par
exemple $\sln$), pour lequel on r\'eserve la notation $G'$. La lissit\'e de $G'$ nous dit par ailleurs que la $k$-vari\'et\'e $V:=G\backslash G'$ est lisse
et g\'eom\'etriquement int\`egre (cf. \cite[VI$_\text{B}$, Prop. 9.2]{SGA3I}) et correspond \`a un espace homog\`ene sous $G'$.

Pour une extension de corps $L/k$ et pour $V$ une $k$-vari\'et\'e, on note $V_L:=V\times_k L$ la $L$-vari\'et\'e obtenue par changement de base.
Pour toute vari\'et\'e $V$, on note $H^i(V,\cdot)$ les groupes de cohomologie \'etale et $H^i(k,\cdot)$ les groupes de cohomologie galoisienne 
classiques (qui co\"\i ncident avec les groupes de cohomologie \'etale pour $V=\spec k$). Pour $i=1$, on peut aussi d\'efinir des ensembles de
cohomologie non ab\'elienne, lesquels seront toujours not\'es $H^1(V,\cdot)$ et $H^1(k,\cdot)$ comme dans le cas ab\'elien (ces deux ensembles
co\"\i ncident aussi pour $V=\spec k$). Si l'on \'ecrit $\gm$ pour le groupe multiplicatif, le groupe de Brauer $\br V$ est d\'efini comme le groupe
$H^2(V,\gm)$. Pour $n> 1$ premier \`a $p$, on note $\mu_n$ le groupe alg\'ebrique (fini et
lisse) des racines $n$-i\`emes de l'unit\'e. La $n$-torsion du groupe de Brauer $\br V$ est alors donn\'ee par un quotient de $H^2(V,\mu_n)$. On note
aussi, pour un groupe $A$ ab\'elien et de torsion (par exemple $\br V$), $A\{p'\}$ son sous-groupe de torsion premi\`ere \`a $p$. Pour les tores, on
utilise la notation $\asd{}{n}{}{}{T}$ pour d\'esigner leur $n$-torsion.\\

Passons aux rappels. Pour la d\'efinition du groupe de Brauer non ramifi\'e $\brnr V$ d'une $k$-vari\'et\'e $V$, on renvoie \`a \cite{ColliotSantaBarbara}
ou encore \`a \cite{ColliotSansucChili}. On se limite \`a rappeler que, lorsque l'on dispose d'une compactification lisse $X$ de $V$, on a l'\'egalit\'e
$\brnr V\{p'\}=\br X\{p'\}$, o\`u $p$ correspond \`a la caract\'eristique du corps $k$, (ce sont des applications du th\'eor\`eme de puret\'e de Grothendieck, cf.
\cite[Thm. 4.1.1, Prop. 4.2.3]{ColliotSantaBarbara}). Le groupe de Brauer alg\'ebrique est d\'efini comme
\[\bral V:=\frac{\ker[\br V\to\br V_{\bar k}]}{\im[\br k\to\br V]}.\]
On note aussi $\brun V:=\ker[\br V\to\br V_{\bar k}]$ d'o\`u $\bral V=\brun V/\br k$ lorsqu'on note par abus $\br k$ sa propre image dans $\br V$.
Le groupe de Brauer non ramifi\'e alg\'ebrique $\brnral V$ est alors l'image de $\brnr V\cap \brun V$ dans $\bral V$ par la projection naturelle.

De m\^eme que dans \cite[\S 4]{GLABrnral}, on rappelle que pour une $k$-vari\'et\'e $V$ v\'erifiant $V(k)\neq\emptyset$ et $\bar k[V]^*/\bar k^*=1$,
le groupe de Brauer alg\'ebrique $\bral V$ admet une description en termes du groupe de Picard g\'eom\'etrique $\pic V_{\bar k}$ (cf. \cite[Lem. 6.3]{Sansuc81}) :
\begin{equation}\label{equation brnr pour V}
\bral V\xrightarrow{\sim} H^1(k,\pic V_{\bar{ k}}).
\end{equation}
De plus, on a de m\^eme pour une $k$-compactification lisse $X$ de $V$ :
\begin{equation}\label{equation brnr pour X}
\bral X\xrightarrow{\sim} H^1(k,\pic X_{\bar{ k}}).
\end{equation}
Enfin, d'apr\`es ce qui a \'et\'e dit ci-dessus, $\bral X\{p'\}$ correspond \`a la partie non ramifi\'ee (et de torsion premi\`ere \`a $p$) du groupe de Brauer
alg\'ebrique de $V$ que l'on note $\brnral V\{p'\}$.\\

Soit maintenant $G$ un $k$-groupe alg\'ebrique lin\'eaire lisse plong\'e dans $G'$ un $k$-groupe alg\'ebrique semi-simple simplement connexe et
soit $V:=G\backslash G'$ la $k$-vari\'et\'e quotient. Le groupe $G'$ \'etant semi-simple, on a que $\bar k[G']^*/\bar k^*=1$ (c'est le lemme de Rosenlicht,
cf. \cite[Lem. 6.5]{Sansuc81}). Ceci nous dit que $\bar k[V]^*/\bar k^*=1$, puisque toute fonction inversible sur $V_{\bar k}$ donne une fonction inversible
sur $G'_{\bar k}$. On en d\'eduit que l'on a la description du groupe $\bral V$ donn\'ee par l'isomorphisme \eqref{equation brnr pour V} ci-dessus.
On conclut cette section en rappelant (cf. toujours \cite[\S 4]{GLABrnral}) que $\pic V_{\bar k}=\hat G=\hat G^\ab$, d'o\`u l'on voit que
$\bral V=H^1(k,\pic V_{\bar k})$ ne d\'epend que de $G$, et m\^eme plus pr\'ecis\'ement de son ab\'elianis\'e. De plus, puisque dans tout ce texte les groupes
finis seront d'ordre premier \`a $p$, le groupe $\hat G^\ab$ n'aura pas de $p$-torsion. Alors, d\`es que les extensions d\'eployant la partie torique
de $G^\ab$ seront d'ordre premier \`a $p$, le groupe $\bral V$ n'aura pas non plus de $p$-torsion. On se permettra alors dans ces cas d'utiliser la notation
$\brnral V$ au lieu de $\brnral V\{p'\}$.

\section{Formules sur un corps fini}\label{section formules corps fini}
Soit maintenant $k$ un corps fini de cardinal $q$ une puissance de $p$, soit $G$ un $k$-groupe alg\'ebrique lin\'eaire \`a composante connexe neutre $G^\circ$ r\'eductive et
soit $F=G/G^\circ$. On rappelle que, lorsque l'ordre de $F$ est premier \`a $p$, le cup-produit nous donne
l'accouplement parfait suivant (cf. \cite[Thm. 7.2.9]{NSW})
\[H^1(k((t)),\hat G^\ab)\times H^1(k((t)),G^\ab)\to H^2(k((t)),\gm)=\br k((t))\cong\q/\z.\]
D'autre part, en consid\'erant la remarque \`a la fin de la section 4.1 de \cite{GLABrnral}, on sait que l'on peut compter sur le r\'esultat suivant :

\begin{pro}\label{proposition cohomologique generale}
Soit $G$ un $k$-groupe alg\'ebrique lin\'eaire \`a composante connexe neutre $G^\circ$ r\'eductive et tel que $F=G/G^\circ$ est d'ordre $n$ premier \`a $q$. Soit
$V=G\backslash G'$ pour un plongement de $G$ dans $G'$ semi-simple simplement connexe. En identifiant $\bral V$ avec $H^1(k,\hat G^\ab)$ comme dans
l'\'equation \eqref{equation brnr pour V}, le groupe de Brauer non ramifi\'e alg\'ebrique $\brnral V\{p'\}$ de $V$ est donn\'e par les \'el\'ements
$\alpha\in H^1(k,\hat G^\ab)\{p'\}$ v\'erifiant la propri\'et\'e suivante :

Pour toute extension finie $k'$ de $k$, l'image $\alpha'$ de $\alpha$ dans $H^1(k'((t)),\hat G^\ab)$ est orthogonale au sous-ensemble
$\im[H^1(k'((t)),G)\to H^1(k'((t)),G^\ab)]$ pour l'accouplement ci-dessus.
\end{pro}

Maintenant, si l'on se donne une inclusion $H\hookrightarrow G^\torf$, on a clairement un $k$-morphisme induit $H^\ab\to G^\ab$ car la projection $G\to G^\ab$
se factorise par $G^\torf$. Via ce morphisme, l'accouplement ci-dessus en induit un autre (pas forc\'ement parfait) qui lui est compatible. C'est-\`a-dire que
l'on a le diagramme commutatif
\[\xymatrix{
H^1(k((t)),H^\ab)\times H^1(k((t)),\hat G^\ab)  \ar[r] \ar@<-15mm>[d] \ar@{=}@<15mm>[d] & \br k((t)) \ar@{=}[d] \\
H^1(k((t)),G^\ab)\times H^1(k((t)),\hat G^\ab)  \ar[r] & \br k((t)).
}\]

\begin{thm}\label{theoreme cohomologique}
Soit $G$ un $k$-groupe alg\'ebrique lin\'eaire \`a composante connexe neutre $G^\circ$ r\'eductive et tel que $F=G/G^\circ$ est d'ordre $n$ premier \`a $q$. Soit
$V=G\backslash G'$ pour un plongement de $G$ dans $G'$ semi-simple simplement connexe. On suppose que le tore $G^\tor=(G^\torf)^\circ$ est d\'eploy\'e par
une extension $L/k$ de degr\'e $d$ premier \`a $q$ et l'on note $m=nd$. On suppose aussi qu'il existe un $k$-sous-groupe fini $H_0$ de $G^\torf$ d'ordre premier
\`a $q$ tel que la compos\'ee $H_0\to G^\torf\to F$ est surjective et l'on note $H$ le sous-groupe (fini aussi) de $G^\torf$ engendr\'e par $H_0$ et par
$\phi_m^{-1}(H_0\cap G^\tor)$, o\`u $\phi_m$ d\'esigne la multiplication par $m$ dans $G^\tor$.

Alors, en identifiant $\bral V$ avec $H^1(k,\hat G^\ab)$, le groupe de Brauer non ramifi\'e alg\'ebrique $\brnral V$ de $V$ est donn\'e par les \'el\'ements
$\alpha\in H^1(k,\hat G^\ab)$ v\'erifiant la propri\'et\'e suivante :

Pour toute extension finie $k'$ de $k$, l'image $\alpha'$ de $\alpha$ dans $H^1(k'((t)),\hat G^\ab)$ est orthogonale au sous-ensemble
$\im[H^1(k'((t)),H)\to H^1(k'((t)),H^\ab)]$ pour l'accouplement ci-dessus.
\end{thm}

\begin{proof}
La proposition \ref{proposition Gille_Ch_R} nous dit que l'application
\[H^1(k'((t)),H)\to H^1(k'((t)),G),\]
est surjective. En rappelant alors la compatibilit\'e des accouplement ci-dessus, on voit que l'on a le diagramme commutatif suivant :
\[\xymatrix{
H^1(k'((t)),H) \ar[r] \ar@{->>}[d] & H^1(k'((t)),H^\ab) \times H^1(k'((t)),\hat G^\ab)  \ar[r] \ar@{=}@<15mm>[d] \ar@<-15mm>[d] & \br k'((t)) \ar@{=}[d] \\
H^1(k'((t)),G) \ar[r] & H^1(k'((t)),G^\ab) \times H^1(k'((t)),\hat G^\ab)  \ar[r] & \br k'((t)).
}\]
De ce diagramme il vient de fa\c con imm\'ediate que, pour $\alpha\in H^1(k,\hat G^\ab)$, le fait que son image dans $H^1(k'((t)),\hat G^\ab)$ soit orthogonale
\`a l'image de $H^1(k'((t)),G)$ dans $H^1(k'((t)),G^\ab)$ est \'equivalent au fait d'\^etre orthogonale \`a l'image de $H^1(k'((t)),H)$ dans $H^1(k'((t)),H^\ab)$
pour le deuxi\`eme accouplement. La proposition \ref{proposition cohomologique generale} nous permet alors de conclure.
\end{proof}

De ce th\'eor\`eme, et gr\^ace aux m\'ethodes d\'evelopp\'ees dans \cite[\S4.2]{GLABrnral}, on d\'eduit une formule alg\'ebrique pour le
groupe $\brnral V$. Pour ce faire, on rappelle des d\'efinitions donn\'ees dans \cite[\S4.2]{GLABrnral}

\begin{defi}
Soit $H$ un $k$-groupe fini d'ordre premier \`a $q$. Soit $s\in\Gamma_k$ l'\'el\'ement correspondant au $q$-Frobenius (i.e. l'application
$\bar k\to \bar k:x\mapsto x^q$). On d\'efinit une application bijective $\varphi_{q,H}:H(\bar k)\to H(\bar k)$ par la formule
\[\varphi_{q,H}(b)=\asd{s^{-1}}{}{q}{}{b}.\]
Elle sera not\'ee $\varphi_q$ lorsqu'il n'y aura pas d'ambig\"uit\'e sur le groupe $H$. On remarque par ailleurs que l'action de $s^{-1}$ et l'\'el\'evation
\`a la $q$-i\`eme puissance commutent, donc la notation n'est pas abusive. L'application induite par $\varphi_q$ sur $H^\ab(\bar k)$ est un
automorphisme que l'on note toujours $\varphi_q$ par abus.

Pour tout $b\in H(\bar k)$, on d\'efinit $n_{b}$ comme le plus petit entier strictement positif tel que $\varphi_q^{n_b}(b)$ soit conjugu\'e \`a $b$.
On d\'efinit alors l'application \emph{$q$-norme} $N_{q,H}:H(\bar k)\to H^\ab(\bar k)$ comme le produit
\[N_{q,H}(b)=\prod_{i=0}^{n_{b}-1}\overline{\varphi_q^i(b)}=\prod_{i=0}^{n_{b}-1}\varphi_q^i(\bar b),\]
o\`u ``$\overline{\phantom{\varphi(b)}}$'' repr\'esente la projection naturelle de $H(\bar k)$ sur $H^\ab(\bar k)$. De m\^eme, elle sera not\'ee $N_q$
lorsqu'il n'y aura pas d'ambig\"uit\'e sur le groupe $H$.
\end{defi}

Avec ces d\'efinitions, on peut \'enoncer la formule alg\'ebrique de la fa\c con suivante.

\begin{thm}\label{theoreme formule sur fq}
Soit $G$ un $k$-groupe alg\'ebrique lin\'eaire \`a composante connexe neutre $G^\circ$ r\'eductive et tel que $F=G/G^\circ$ est d'ordre $n$ premier \`a $q$. Soit
$V=G\backslash G'$ pour un plongement de $G$ dans $G'$ semi-simple simplement connexe. On suppose que le tore $G^\tor=(G^\torf)^\circ$ est d\'eploy\'e par
une extension $L/k$ de degr\'e $d$ premier \`a $q$ et l'on note $m=nd$. On suppose aussi qu'il existe un $k$-sous-groupe fini $H_0$ de $G^\torf$ d'ordre premier \`a
$q$ tel que la compos\'ee $H_0\to G^\torf\to F$ est surjective et l'on note $H$ le sous-groupe (fini aussi) de $G^\torf$ engendr\'e par $H_0$ et par
$\phi_m^{-1}(H_0\cap G^\tor)$, o\`u $\phi_m$ d\'esigne la multiplication par $m$ dans $G^\tor$.

Alors, en identifiant $\bral V$ avec $H^1(k,\hat G^\ab)$, le groupe de Brauer non ramifi\'e alg\'ebrique $\brnral V$ de $V$ est donn\'e par les \'el\'ements
$\alpha\in H^1(k,\hat G^\ab)$ tels que, pour $a\in Z^1(k,\hat G^\ab)$ un cocycle (quelconque) repr\'esentant $\alpha$, on a
\[a_{s}(N_{q,H}(b))=1\quad \forall\, b\in H(\bar k).\]
\end{thm}

On remarque que la formule a bien un sens. En effet, $a_s$ est un \'el\'ement de $\hat G^\ab$, donc on peut l'\'evaluer en $N_q(b)\in H^\ab(\bar k)$ en poussant
cet \'el\'ement dans $G^\ab$, ce qui donne un \'el\'ement dans $\mu_h$, o\`u $h$ est l'exposant de $H$.

\begin{proof}
Il s'agit de d\'emontrer que la propri\'et\'e cohomologique du th\'eor\`eme \ref{theoreme cohomologique} est \'equivalente \`a celle de l'\'enonc\'e
ci-dessus. Autrement dit, il suffit de montrer que, pour $\alpha=[a]\in H^1(k,\hat G^\ab)$, si l'on note $r'(\alpha)$ sa restriction \`a
$H^1(k'((t)),\hat G^\ab)$, alors on a
\begin{gather*}
r'(\alpha) \perp\im[H^1(k'((t)),H)\to H^1(k'((t)),H^\ab)]\quad\forall\, k'/k\,\text{finie}\\
\Updownarrow\\
a_{s}(N_{q,H}(b))=1\quad \forall\, b\in H(\bar k).
\end{gather*}
D\'emontrons alors cette \'equivalence. L'application $H^\ab\to G^\ab$ nous donne une application duale $\hat G^\ab\to \hat H^\ab$. L'accouplement
entre $\hat G^\ab$ et $H^\ab$ que l'on a utilis\'e pour \'enoncer le th\'eor\`eme \ref{theoreme cohomologique} est \'evidemment compatible avec celui (naturel) entre $\hat H^\ab$ et
$H^\ab$. Il en est par cons\'equent de m\^eme au niveau du cup produit des $H^1$. Autrement dit, on a le diagramme commutatif suivant :
\[\xymatrix{
H^1(k((t)),H^\ab)\times H^1(k((t)),\hat G^\ab)  \ar[r] \ar@<15mm>[d] \ar@{=}@<-15mm>[d] & \br k((t)) \ar@{=}[d] \\
H^1(k((t)),H^\ab)\times H^1(k((t)),\hat H^\ab)  \ar[r] & \br k((t)).
}\]
Si l'on note alors $a_H$ l'image de $a$ dans $Z^1(k,\hat H^\ab)$, $\alpha_H=[a_H]$ et $r'_H$
la restriction $H^1(k,\hat H^\ab)\to H^1(k'((t)),\hat H^\ab)$, il est \'evident que l'on a les deux \'equivalences suivantes :
\begin{gather*}
r'_H(\alpha_H) \perp\im[H^1(k'((t)),H)\to H^1(k'((t)),H^\ab)]\quad\forall\, k'/k\,\text{finie}\\
\Updownarrow\\
r'(\alpha)\perp\im[H^1(k'((t)),H)\to H^1(k'((t)),H^\ab)]\quad\forall\, k'/k\,\text{finie},
\end{gather*}
et
\begin{gather*}
(a_H)_{s}(N_{q,H}(b))=1\quad \forall\, b\in H(\bar k)\\
\Updownarrow\\
a_{s}(N_{q,H}(b))=1\quad \forall\, b\in H(\bar k).
\end{gather*}
On voit alors qu'il suffit de d\'emontrer
\begin{gather*}
r'_H(\alpha_H) \perp\im[H^1(k'((t)),H)\to H^1(k'((t)),H^\ab)]\quad\forall\, k'/k\,\text{finie}\\
\Updownarrow\\
(a_H)_{s}(N_{q,H}(b))=1\quad \forall\, b\in H(\bar k).
\end{gather*}
Or, cette derni\`ere \'equivalence d\'ecoule des th\'eor\`emes 4.1 et 4.5 dans \cite{GLABrnral}, lesquels calculent tous les deux le groupe $\brnral W$
pour $W=H\backslash H'$ avec $H'$ semi-simple simplement connexe, le premier avec la formule du dessus, le deuxi\`eme avec la formule en dessous.
\end{proof}

On conclut cette section en notant que les th\'eor\`emes \ref{theoreme cohomologique} et \ref{theoreme formule sur fq} nous
permettent de comparer explicitement le groupe $\brnral V$ avec le groupe de Brauer non ramifi\'e alg\'ebrique d'un espace homog\`ene \`a
stabilisateur fini. En effet, soient $G$, $G'$, $H_0$ et $H$ comme dans l'\'enonc\'e des th\'eor\`emes \ref{theoreme cohomologique} et
\ref{theoreme formule sur fq}. Consid\'erons une inclusion $H\hookrightarrow H'$ avec $H'$ semi-simple simplement connexe et posons
$W=H\backslash H'$. On sait d\'ej\`a que le groupe $\bral W$ s'identifie \`a $H^1(k,\hat H^\ab)$. En appliquant alors ces th\'eor\`emes au calcul de
$\brnral W$ en prenant le groupe $H$ tout entier comme son sous-groupe fini (cas o\`u $H^\torf=H$ et $H^\tor=1$), on voit que l'on a de fa\c con 
\'evidente le corollaire suivant.

\begin{cor}\label{corollaire comparaison de Brnral entre G et H corps finis}
Soit $\alpha_H$ l'image de $\alpha\in\bral V=H^1(k,\hat G^\ab)$ dans $H^1(k,\hat H^\ab)=\bral W$. Alors $\alpha\in\brnral V$ si et seulement si 
$\alpha_H\in\brnral W$. En particulier, on a $\ker[H^1(k,\hat G^\ab)\to H^1(k,\hat H^\ab)]\subset\brnral V$.\qed
\end{cor}

\section{Formule sur un corps de caract\'eristique 0}\label{section formules caracteristique 0}
Dans toute cette section, $k$ d\'esigne un corps de caract\'eristique $0$. De la m\^eme fa\c con que dans \cite{GLABrnral}, on retrouve ici une formule
pour le groupe $\brnral V$ o\`u $V=G\backslash G'$ en nous ramenant au cas des corps finis et en utilisant la formule d\'evelopp\'ee dans la section
pr\'ec\'edente. On rappelle que cette m\'ethode a d\'ej\`a \'et\'e utilis\'ee par Colliot-Th\'el\`ene et Kunyavski\u\i\, (cf. \cite{ColliotKunyavskii})
dans le cas des espaces principaux homog\`enes, i.e. lorsque $G$ est trivial, et par Borovoi, Demarche et Harari (cf. \cite{BDH}) dans le cas o\`u
$G$ est connexe ou commutatif (ou encore plus g\'en\'eralement de type ``ssumult'') et avec des groupes $G'$ plus g\'en\'eraux.

On commence en rappelant quelques d\'efinitions analogues \`a celles donn\'ees dans la section \ref{section formules corps fini}.

\begin{defi}
Soit $H$ un $k$-groupe fini d'exposant $h$. Soit $\zeta_h\in\bar k$ une racine primitive $h$-i\`eme de l'unit\'e.
On d\'efinit le morphisme $q:\Gamma_k\to(\z/h\z)^*$ par la formule suivante :
\[\forall\,\sigma\in\Gamma_k,\quad\asd{\sigma}{}{}{h}{\zeta}=\zeta_h^{q(\sigma)}.\]
C'est le caract\`ere cyclotomique modulo $h$, cf. \cite[Def. 7.3.6]{NSW}.

On d\'efinit ensuite, pour tout $\sigma\in\Gamma_k$, l'application $\varphi_{\sigma,H}:H(\bar k)\to H(\bar k)$ par la formule
\[\varphi_{\sigma,H}(b)=\asd{\sigma^{-1}}{}{q(\sigma)}{}{b},\]
laquelle sera not\'ee $\varphi_\sigma$ lorsqu'il n'y aura pas d'ambig\"uit\'e sur le groupe $H$. On note abusivement $\varphi_\sigma$ aussi
l'application induite sur $H^\ab(\bar k)$ qui en est un automorphisme. On remarque par ailleurs que l'action de $\sigma^{-1}$ et l'\'el\'evation \`a la
$q(\sigma)$-i\`eme puissance commutent, donc la notation n'est pas abusive.

Enfin, pour $\sigma\in\Gamma_k$ et $b\in H(\bar k)$, on d\'efinit $n_{\sigma,b}$ comme le plus petit entier strictement positif tel que
$\varphi^{n_{\sigma,b}}_\sigma(b)$ soit conjugu\'e \`a $b$. On d\'efinit alors l'application \emph{$\sigma$-norme}
$N_{\sigma,H}:H(\bar k)\to H^\ab(\bar k)$ par la formule suivante :
\[N_{\sigma,H}(b):=\prod_{i=0}^{n_{\sigma,b}-1}\overline{\varphi_\sigma^i(b)}=\prod_{i=0}^{n_{\sigma,b}-1}\varphi_\sigma^i(\bar b).\]
De m\^eme, $N_{\sigma,H}$ sera not\'ee $N_\sigma$ lorsqu'il n'y aura pas d'ambig\"uit\'e sur le groupe $H$.
\end{defi}

Avec ces d\'efinitions, on peut \'enoncer la formule pour $\brnral V$.

\begin{thm}\label{theoreme formule en car 0}
Soit $G$ un $k$-groupe alg\'ebrique lin\'eaire et soit $V=G\backslash G'$ pour un plongement de $G$ dans $G'$ semi-simple simplement connexe.
Soit $L/k$ une extension d\'eployant le tore $G^\tor=(G^\torf)^\circ$. On note $d$ le degr\'e de cette extension, $n$ l'ordre du groupe fini
$F=G/G^\circ$ et $m=nd$. On suppose aussi qu'il existe un $k$-sous-groupe fini $H_0$ de $G^\torf$ tel que
la compos\'ee $H_0\to G^\torf\to F$ est surjective et l'on note $H$ le sous-groupe (fini aussi) de $G^\torf$ engendr\'e par $H_0$ et par
$\phi_m^{-1}(H_0\cap G^\tor)$, o\`u $\phi_m$ d\'esigne la multiplication par $m$ dans $G^\tor$.

Alors, en identifiant $\bral V$ avec $H^1(k,\hat G^\ab)$, le groupe de Brauer non ramifi\'e alg\'ebrique $\brnral V$ de $V$ est donn\'e par les \'el\'ements
$\alpha\in H^1(k,\hat G^\ab)$ tels que, pour $a\in Z^1(k,\hat G^\ab)$ un cocycle (quelconque) repr\'esentant $\alpha$, on a
\begin{equation}\label{equation formule brnral en car 0}
a_\sigma(N_{\sigma,H}(b))=1\quad \forall\, b\in H(\bar k),\forall\, \sigma\in\Gamma_k.
\end{equation}
\end{thm}

On commence en remarquant que la proposition \ref{proposition remarque Colliot} dans l'appendice, issue d'une remarque faite par Colliot-Th\'el\`ene, nous
dit que l'on peut se restreindre au cas o\`u la composante connexe neutre $G^\circ$ de $G$ est r\'eductive. En effet, cette proposition nous dit que, si l'on plonge le
quotient $G_0=G/R_\mathrm{u}(G)$ dans un autre groupe $G'_0$ semi-simple simplement connexe et que l'on note $W=G_0\backslash G'_0$, alors le
groupe $\brnr V$ est isomorphe au groupe $\brnr W$. D'autre part, il est \'evident que l'on a $\bral V\cong\bral W$ car ils sont tous les deux isomorphes \`a
$H^1(k,\hat G^\ab)$ puisque $\hat G^\ab=\hat G=\hat G_0$. On en d\'eduit l'\'egalit\'e $\brnral V\cong\brnral W$ et l'on voit alors qu'il suffit de calculer
le groupe $\brnral W$.\\

On suppose alors d\'esormais que la composante connexe neutre $G^\circ$ de $G$ est r\'eductive. Soit $X$ une $k$-compactification lisse de $V$. On rappelle que de
la m\^eme fa\c con que l'on identifie $\bral V$ avec $H^1(k,\hat G^\ab)$, on peut identifier $\brnral V=\bral X$ avec $H^1(k,\pic X_{\bar k})$ et que l'on a
le r\'esultat suivant (cf. \cite[Lem. 4.16]{GLABrnral}) :

\begin{lem}\label{lemme brnral <<egal au sha1>>}
Soit $G$ un $k$-groupe alg\'ebrique lin\'eaire plong\'e dans $G'$ semi-simple simplement connexe et soit $V=G\backslash G'$. Soit $X$ une $k$-compactification lisse de
$V$. Soit enfin $\alpha\in H^1(k,\hat G^\ab)$. Supposons que pour tout sous-groupe pro-cyclique $\gamma$ de $\Gamma_k$, l'image de $\alpha$ dans
$H^1(\gamma,\hat G^\ab)$ provient de $H^1(\gamma,\pic X_{\bar k})$. Alors $\alpha$ provient de $H^1(k,\pic X_{\bar k})$.
\end{lem}

Compte tenu de ce lemme, on voit facilement que pour d\'emontrer le th\'eor\`eme \ref{theoreme formule en car 0} il suffit de d\'emontrer, pour
$\alpha\in H^1(k,\hat G^\ab)$, $\sigma\in\Gamma_k$ et $\gamma$ le sous-groupe ferm\'e pro-cyclique engendr\'e par $\sigma$, la proposition suivante :

\begin{pro}\label{proposition intermediaire theoreme formule en car 0}
Avec les notations ci-dessus, la restriction $\alpha'\in H^1(\gamma,\hat G^\ab)$ de $\alpha$ provient de $H^1(\gamma,\pic X_{\bar k})$ si et seulement si,
pour $a\in Z^1(k,\hat G^\ab)$ un cocycle repr\'esentant $\alpha$, on a
\[a_\sigma(N_{\sigma,H}(b))=1\quad \forall\, b\in H.\]
\end{pro}

En effet, on peut voir que la restriction $\alpha'$ est clairement repr\'esent\'ee par le cocycle $a|_{\gamma}\in Z^1(\gamma,M)$ et donc $a_\sigma$
n'est rien d'autre que $(a|_\gamma)_\sigma$. La d\'emarche consistera alors \`a trouver un corps fini $\f$ tel que l'on puisse regarder $H(\bar k)$,
$\hat G^\ab$ et $\hat H^\ab$ comme des $\Gamma_\f$-groupes dont l'action du Frobenius co\"\i ncide avec celle de $\sigma$. On aura alors des
morphismes canoniques
$H^1(\gamma,\hat G^\ab)\to H^1(\f,\hat G^\ab)$ et $H^1(\gamma,\pic X_{\bar k})\xrightarrow{\sim} H^1(\f,\pic X_{\bar\f})$ et des \'egalit\'es
$\varphi_{q_0,H}=\varphi_{\sigma,H}$ et $N_{q_0,H}=N_{\sigma,H}$,
o\`u $q_0$ est le cardinal de $\f$. Le r\'esultat d\'ecoulera alors du th\'eor\`eme \ref{theoreme formule sur fq}.\\

Pour ce faire, on suit la d\'emarche de Colliot-Th\'el\`ene et Kunyavski\u\i , laquelle est explicit\'ee en d\'etails dans \cite[\S4.3.1]{GLABrnral} et que
l'on r\'esume ici.

\begin{proof}[D\'emonstration de la proposition \ref{proposition intermediaire theoreme formule en car 0}]
La premi\`ere \'etape consiste \`a se r\'eduire au cas o\`u $k$ est de type fini sur $\q$. On utilise notamment le fait que toutes les vari\'et\'es et
morphismes consid\'er\'es sont de type fini et que le groupe de Picard $\pic X_{\bar k}$ est sans torsion, raison pour laquelle le groupe
$H^1(k,\pic X_{\bar k})$ s'annule par restriction \`a une extension finie $\tilde k/k$. On suppose aussi pour la suite que $\tilde k$ contient les racines
$h$-i\`emes de l'unit\'e.

Ensuite, puisque $k$ est de type fini sur $\q$, on sait que l'on peut trouver :
\begin{enumerate}
\item Un anneau int\`egre et r\'egulier $A$ de type fini sur $\z$ et de corps de fractions $k$, tel que la fermeture int\'egrale $\tilde A$ dans $\tilde k$
soit finie, \'etale et galoisienne sur $A$ de groupe de Galois $\tilde\Gamma=\gal(\tilde k/k)$.
\item Des $A$-sch\'emas $\mathcal{G}$, $\cal{H}_0$, $\cal{H}$, $\cal{G}'$, $\mathcal{V}$ et $\mathcal{X}$ de type fini, de fibres g\'en\'eriques
respectives $G$, $H_0$, $H$, $G'$, $V$ et $X$ et tels que :
\begin{itemize}
\item $\cal{G}'$ est un $A$-sch\'ema en groupes semi-simple simplement connexe (au sens de \cite[XIX.2.7, XXII.4.3.3]{SGA3III}), $\mathcal{G}$ est
un sous-$A$-sch\'ema en groupes lisse de $\cal{G}'$ ;
\item $\cal{H}_0$ est un sous-$A$-sch\'ema en groupes fini et lisse de $\cal{G}^\torf$ et $\mathcal{H}$ est le sous-groupe de $\cal{G}^\torf$
engendr\'e par $\cal{H}_0$ et $\phi_m^{-1}(\cal{H}_0\cap \cal{G}^\tor)$, o\`u $\phi_m$ d\'esigne la multiplication par $m$ dans $\cal{G}^\tor$.
\item $\mathcal{V}=\mathcal{G}\backslash \cal{G}'$ est lisse sur $A$, $\mathcal{X}$ est propre et lisse sur $A$ et on a une $A$-immersion ouverte
$\mathcal{V}\hookrightarrow\mathcal{X}$ \'etendant $V\hookrightarrow X$ ;
\item pour tout point $x$ de $\spec A$ de corps r\'esiduel $\kappa_x$, la fibre $\mathcal{X}_x\subset\mathcal{X}$ est une compactification lisse de
$\mathcal{V}_x\subset\mathcal{V}$ sur $\kappa_x$.
\end{itemize}
\end{enumerate}
En notant que $X$ est unirationnelle et que $k$ est un corps de caract\'eristique $0$, on voit que $H^1(X,\cal{O}_X)=H^2(X,\cal{O}_X)=0$ et, quitte \`a
restreindre $\spec A$, on peut supposer que $H^1(\mathcal{X}_x,\cal{O}_{\mathcal{X}_x})=H^2(\mathcal{X}_x,\cal{O}_{\mathcal{X}_x})=0$ pour tout
point $x$ de $\spec A$. Alors, d'apr\`es \cite[Cor. 2.7]{GrothendieckPicard}, on trouve des isomorphismes de sp\'ecialisation
\[\pic X\xrightarrow{\sim}\pic \mathcal{X}_x\]
pour tout point $x\in\spec A$ et, quitte \`a r\'eduire encore $\spec A$, des isomorphismes 
\[\pic X_{\tilde k}\xrightarrow{\sim}\pic \tilde{\mathcal{X}}_{\tilde x}\]
compatibles avec les actions galoisiennes correspondantes pour tout point $\tilde x\in\spec\tilde A$, o\`u
$\tilde{\mathcal{X}}=\mathcal{X}\times_A\tilde A$.

Si l'on note alors $\tilde\gamma$ l'image de $\gamma\subset\Gamma_k$ dans $\tilde\Gamma$ et que l'on consid\`ere le sous-anneau
$A_1:=\tilde A^{\tilde\gamma}$ des \'el\'ements $\tilde\gamma$-invariants de $\tilde A$, le morphisme $\rho:\spec\tilde A\to\spec A_1$ est un
rev\^etement \'etale cyclique dont le groupe de Galois est $\tilde\gamma$ et avec $\tilde A$ et $A_1$ des anneaux int\`egres, r\'eguliers et de type fini
sur $\z$. La version g\'eom\'etrique du th\'eor\`eme de \v Cebotarev (cf. \cite[Thm. 7]{SerreZeta}) nous dit alors qu'il existe une infinit\'e de points
ferm\'es $x_1$ de $\spec A_1$ tels que la fibre en $x_1$ de $\rho$ est un point ferm\'e $\tilde x$ de $\spec \tilde A$. Ceci veut dire que si l'on note
$\f$ le corps r\'esiduel en $x_1$ et $\mathbb{E}$ le corps r\'esiduel en $\tilde x$, on a une extension de corps r\'esiduels $\mathbb{E}/\f$ de groupe
de Galois $\tilde\gamma$ et par suite, en notant $\mathcal{X}_1=\mathcal{X}\times_A A_1$ et $Y$ la $\f$-vari\'et\'e $(\mathcal{X}_1)_{x_1}$, un
isomorphisme
\[H^1(\tilde\gamma,\pic X_{\tilde k})\xrightarrow{\sim} H^1(\tilde\gamma,\pic Y_{\mathbb{E}}).\]
De plus, d'apr\`es le lemme 6.3(iii) de \cite{Sansuc81}, on a que $\pic Y_{\mathbb{E}}\cong\pic Y_{\bar\f}$ car $\Gamma_{\bb E}$ agit trivialement sur
$\pic Y_{\bar\f}$.\\

De tout cela on obtient l'isomorphisme
\[H^1(\gamma,\pic X_{\bar k})\xrightarrow{\sim} H^1(\gamma,\pic Y_{\bar\f}),\]
et en particulier, le diagramme commutatif suivant :
\begin{equation}\label{equation diagramme specialisation des picard}
\xymatrix{
H^1(\gamma,\pic X_{\bar k}) \ar@{^{(}->}[r] \ar[d]^{\sim} & H^1(\gamma,\hat G^\ab) \ar@{=}[d] \\
H^1(\gamma,\pic Y_{\bar\f}) \ar@{^{(}->}[r] & H^1(\gamma,\hat{\cal G}_{x_1}^\ab) \\
}
\end{equation}
o\`u $\hat G^\ab=\hat{\cal G}_{x_1}^\ab$ repr\'esente le groupe des caract\`eres (g\'eom\'etriques) de $G$ (ou encore de $\cal G_{x_1}$). Par ailleurs,
l'action de $\gamma$ sur $\hat G^\ab$ et sur $\hat{\cal G}_{x_1}^\ab$ \'etant induite par celle sur $\pic X_{\bar k}$ et sur $\pic X_{\bar\f}$
respectivement, on voit qu'il s'agit bien de la m\^eme action.

Enfin, le morphisme naturel $f:\Gamma_\f\to\gamma$ issu de cette construction nous donne les morphismes
\[H^1(\gamma,\hat G^\ab)\hookrightarrow H^1(\f,\hat{\cal G}^\ab_{x_1})\quad\text{et}\quad
H^1(\gamma,\pic Y_{\bar\f})\xrightarrow{\sim}H^1(\f,\pic Y_{\bar\f})=\bral Y,\]
o\`u l'action de $\Gamma_\f$ sur $\hat{\cal G}^\ab_{x_1}$ correspond bien entendu \`a celle induite par $f$ (il en va de m\^eme par ailleurs pour
l'action sur $G(\bar k)=\cal G_{x_1}(\bar \f)$). La premi\`ere fl\`eche est en plus un isomorphisme d\`es que l'ordre de $\sigma$ est infini (car dans
ce cas le morphisme $f$ est lui-m\^eme un isomorphisme).
De ces morphismes et du diagramme \eqref{equation diagramme specialisation des picard} on obtient le diagramme commutatif
\begin{equation}\label{equation diagramme specialisation des picard 2}
\xymatrix{
H^1(\gamma,\pic X_{\bar k}) \ar@{^{(}->}[r] \ar[d]^{\sim} & H^1(\gamma,\hat G^\ab) \ar@{^{(}->}[d] \\
H^1(\f,\pic Y_{\bar\f}) \ar@{^{(}->}[r] & H^1(\f,\hat{\cal G}_{x_1}^\ab) \\
}
\end{equation}
Alors, puisque d'une part $Y$ est une compactification lisse d'un espace homog\`ene sous le $\f$-groupe $\cal {G}'_{x_1}$ \`a stabilisateur
$\cal G_{x_1}$ \`a composante connexe neutre r\'eductive et admettant un sous-groupe fini $\cal H_{x_1}$ de $\cal G^\tor_{x_1}$ avec les bonnes
propri\'et\'es,
et d'autre part puisque l'on peut toujours choisir $x_1$ de fa\c con que l'ordre de $\cal H_{x_1}$ soit premier au cardinal de $\f=\kappa_{x_1}$, on a
enfin le droit d'utiliser le th\'eor\`eme \ref{theoreme formule sur fq}.\\

Reprenons alors un \'el\'ement $\alpha\in H^1(k,\hat G^\ab)$, $a\in Z^1(k,\hat G^\ab)$ repr\'esentant $\alpha$, $\sigma\in\Gamma_k$ et $\gamma$
le sous-groupe pro-cyclique de $\Gamma_k$ engendr\'e par $\sigma$ comme dans la proposition
\ref{proposition intermediaire theoreme formule en car 0}. D'apr\`es le diagramme commutatif \eqref{equation diagramme specialisation des picard 2},
la restriction $\alpha'$ de $\alpha$ dans $H^1(\gamma,\hat G^\ab)$ provient de $H^1(\gamma,\pic X_{\bar k})$ si et seulement si son image
$\tilde\alpha\in H^1(\f,\hat{\cal G}^\ab_{x_1})$ provient de $H^1(\f,\pic Y_{\bar\f})=\bral Y=\brnral (\cal V_{x_1})$. Consid\'erons les cocycles
$a'=a|_{\gamma}$ et $\tilde a=a'\circ f$, qui repr\'esentent respectivement $\alpha'$ et $\tilde\alpha$. On voit alors gr\^ace au th\'eor\`eme
\ref{theoreme formule sur fq} que $\tilde\alpha$ provient de $H^1(\f,\pic Y_{\bar\f})$ si et seulement si
\[\tilde a_{s}(N_{q_0,\cal H_{x_1}}(b))=1\quad\forall\, b\in \cal G_{x_1}(\bar \f)=G(\bar k),\]
o\`u $q_0$ est le cardinal de $\f$ et $s\in\Gamma_\f$ d\'esigne le Frobenius. Or, on a clairement
\[\tilde a_{s}(N_{q_0,\cal H_{x_1}}(b))=a'_{\sigma}(N_{q_0,\cal H_{x_1}}(b))=a_{\sigma}(N_{q_0,\cal H_{x_1}}(b)),\]
d'o\`u l'on voit qu'il suffit de d\'emontrer que $N_{q_0,\cal H_{x_1}}=N_{\sigma,H}$ pour conclure.

D'apr\`es la d\'efinition de $N_{q_0,\cal H_{x_1}}$ et de $N_{\sigma,H}$, il suffit en fait de d\'emontrer l'\'egalit\'e
$\varphi_{q_0,\cal H_{x_1}}=\varphi_{\sigma,H}$. En remarquant alors que les actions de $\gamma$ et $\Gamma_\f$ co\"\i ncident, les d\'efinitions respectives de
$\varphi_{q_0,\cal H_{x_1}}$ et $\varphi_{\sigma,H}$ nous disent que tout ce qu'il faut d\'emontrer est que
\[q(\sigma)\equiv q_0\mod h,\quad\text{ou encore}\quad\zeta_h^{q_0}=\zeta_h^{q(\sigma)},\]
o\`u $h$ est l'exposant de $H$. Or, cela est vrai de fa\c con \'evidente car $\sigma$ agit sur $\mu_{h}$ de la m\^eme fa\c con que
le Frobenius de $\f$ (on rappelle que $\mu_h\subset\tilde k$ par hypoth\`ese), lequel envoie $\zeta_{h}$ en $\zeta_{h}^{q_0}$ par d\'efinition, d'o\`u
\[\zeta_h^{q(\sigma)}=\asd{\sigma}{}{}{h}{\zeta}=\asd{s}{}{}{h}{\zeta}=\zeta_h^{q_0}.\]
Ceci conclut donc la d\'emonstration de la proposition \ref{proposition intermediaire theoreme formule en car 0} et par cons\'equent celle du
th\'eor\`eme \ref{theoreme formule en car 0}.
\end{proof}

\vspace{5mm}

On peut maintenant essayer de tirer des cons\'equences de la formule que l'on vient de d\'evelopper. Tout comme dans la section pr\'ec\'edente, soient
$G$, $G'$, $H_0$ et $H$ comme dans l'\'enonc\'e du th\'eor\`eme \ref{theoreme formule en car 0}. Consid\'erons une inclusion $H\hookrightarrow H'$
avec $H'$ semi-simple simplement connexe et posons $W=H\backslash H'$. Alors $\bral W$ s'identifie \`a $H^1(k,\hat H^\ab)$ et le m\^eme
th\'eor\`eme \ref{theoreme formule en car 0} nous permet de calculer ce groupe, d'o\`u l'on obtient l'analogue du corollaire
\ref{corollaire comparaison de Brnral entre G et H corps finis}.

\begin{cor}\label{corollaire comparaison de Brnral entre G et H car 0}
Soit $\alpha_H$ l'image de $\alpha\in\bral V=H^1(k,\hat G^\ab)$ dans $H^1(k,\hat H^\ab)=\bral W$. Alors $\alpha\in\brnral V$ si et seulement si $\alpha_H\in\brnral W$.
En particulier, on a $\ker[H^1(k,\hat G^\ab)\to H^1(k,\hat H^\ab)]\subset\brnral V$.\qed
\end{cor}

On peut ensuite d\'emontrer que l'on peut expliciter une extension finie du corps de base $k$ sur laquelle toutes les classes du groupe $\brnral V$ disparaissent. Cela
veut dire que l'application de restriction du groupe $\brnral V$ \`a cette extension est triviale, ou encore que toutes les classes peuvent \^etre retrouv\'ees par
inflation, comme il est explicit\'e dans l'\'enonc\'e qui suit.

\begin{pro}\label{proposition brnral se calcule en une extension finie}
Sous les m\^emes hypoth\`eses que le th\'eor\`eme \ref{theoreme formule en car 0}, soient $L/k$ une extension finie galoisienne de $k$
d\'eployant $H$ et $G^\tor$, $h$ l'exposant de $H$ et $\zeta_h\in\bar k$ une racine primitive $h$-i\`eme de l'unit\'e. Alors on a
\[\brnral V\subset H^1(L(\zeta_h)/k,(\hat G^\ab)^{\Gamma_{L(\zeta_h)}}).\]
\end{pro}

\begin{proof}
Soit $\sigma\in\Gamma_{L(\zeta_h)}$, alors il est \'evident que $\sigma$ fixe $\zeta_h$, ce qui nous dit que $q(\sigma)=1$. De plus, puisque $L$ d\'eploie $H$,
$\sigma$ agit de fa\c con triviale sur $H(\bar k)$. On a alors que $\varphi_\sigma$ se r\'eduit \`a l'identit\'e et alors $N_\sigma$ correspond \`a la projection
$H\to H^\ab$.

Soient alors $a\in Z^1(k,\hat G^\ab)$ un cocycle v\'erifiant l'\'egalit\'e \eqref{equation formule brnral en car 0}, $\alpha=[a]\in H^1(k,\hat G^\ab)$ et $\beta$
la restriction de $\alpha$ \`a $H^1(L(\zeta_h),\hat G^\ab)$. Il s'agit de montrer que $\beta=0$, car cela veut dire que $\alpha$ provient par inflation de
$H^1(L(\zeta_h)/k,(\hat G^\ab)^{\Gamma_{L(\zeta_h)}})$. Or, le cocycle $b=a|_{L(\zeta_h)}$ repr\'esentant $\beta$ v\'erifie $b_{\sigma}(\b)=1$ pour tout
$\b\in H^\ab(\bar k)$, ce qui nous dit que l'image de $b_\sigma$ dans $\hat H^\ab(\bar k)$ est nulle. Ceci montre que l'image de $\beta$ dans
$H^1(L(\zeta_h),\hat H^\ab)$ est triviale, i.e. \[\beta\in\ker[H^1(L(\zeta_h),\hat G^\ab)\to H^1(L(\zeta_h),\hat H^\ab)].\]

Soit maintenant $S_G$ le noyau de l'application $G^\ab\to F^\ab$. On affirme qu'il s'agit d'un quotient de $G^\tor$. En effet, on a le diagramme commutatif
\`a lignes et colonnes exactes
\[\xymatrix@=5mm{
& 1 \ar[d] & 1 \ar[d] & 1 \ar[d] \\
1 \ar[r] & D(G^\torf)\cap G^\tor \ar[d] \ar[r] & D(G^\torf) \ar[d] \ar[r] & D(F) \ar[r] \ar[d] & 1 \\
1 \ar[r] & G^\tor \ar[r] \ar[d] & G^\torf \ar[r] \ar[d] & F \ar[r] \ar[d] & 1 \\
1 \ar[r] & S_G \ar[r] & G^\ab \ar[r] \ar[d] & F^\ab \ar[r] \ar[d] & 1 \\
& & 1 & 1.
}\]
Une chasse au diagramme nous montre alors que l'application $G^\tor\to S_G$ est surjective, ce qui nous dit que $S_G$ est un tore. Or, ce tore est clairement
d\'eploy\'e par $L$, ce qui nous dit que $H^1(L(\zeta_h),\hat S_G)=0$. Consid\'erons alors le diagramme commutatif \`a lignes exactes
\[\xymatrix{
1 \ar[r] & S_H \ar[r] \ar[d] & H^\ab \ar[r] \ar[d] & F^\ab \ar[r] \ar@{=}[d] & 1 \\
1 \ar[r] & S_G \ar[r] & G^\ab \ar[r] & F^\ab \ar[r] & 1,
}\]
o\`u $S_H$ est d\'efini comme le noyau de $H^\ab\to F^\ab$. En prenant le dual de Cartier des \'el\'ements de ce diagramme et en consid\'erant les suites exactes
longues de cohomologie associ\'ees, on trouve le diagramme commutatif \`a lignes exactes
\[\xymatrix@C=5mm{
& & H^1(L(\zeta_h),\hat F^\ab) \ar[r] \ar@{=}[d] & H^1(L(\zeta_h),\hat G^\ab) \ar[d] \ar[r] & H^1(L(\zeta_h),\hat S_G) \\
H^0(L(\zeta_h),\hat H^\ab) \ar[r] & H^0(L(\zeta_h),\hat S_H) \ar[r]^\delta & H^1(L(\zeta_h),\hat F^\ab) \ar[r] & H^1(L(\zeta_h),\hat H^\ab).
}\]
Puisque $\beta\in\ker[H^1(L(\zeta_h),\hat G^\ab)\to H^1(L(\zeta_h),\hat H^\ab)]$ et $H^1(L(\zeta_h),\hat S_G)=0$, on a que $\beta$ provient de
$H^1(L(\zeta_h),\hat F^\ab)$ et que son image dans $H^1(L(\zeta_h),\hat H^\ab)$ est triviale, d'o\`u l'on sait qu'elle est dans l'image de $\delta$.
Or, puisque $L(\zeta_h)$ d\'eploie clairement $\hat H^\ab$ (car elle d\'eploie $H$ et $\mu_h$), on a que l'application $\hat H^\ab(L(\zeta_h))\to\hat S_H(L(\zeta_h))$
est surjective, ce qui montre que $\delta$ est triviale et donc que $\beta$ l'est aussi, ce qui conclut.
\end{proof}

\vspace{5mm}

Une application imm\'ediate du corollaire \ref{corollaire comparaison de Brnral entre G et H car 0} est le r\'esultat qui suit, donnant une formule cohomologique
pour le groupe $\brnral V$ lorsque le corps de base est un corps local, i.e. une extension finie de $\q_p$ pour $p$ un nombre premier,
\`a l'image de la proposition 4.23 de \cite{GLABrnral}.

On dira qu'un $k$-groupe fini $G$ est \emph{non ramifi\'e} s'il existe une extension finie non ramifi\'ee $k'/k$ telle que $G_{k'}$ est un $k'$-groupe fini \emph{constant}.

\begin{pro}\label{proposition cohomologique corps local}
Soit $k$ un corps local de caract\'eristique $0$ et caract\'eristique r\'esiduelle $p>0$, $G$ un $k$-groupe alg\'ebrique lin\'eaire et $V=G\backslash G'$ pour un
plongement de $G$ dans $G'$ semi-simple simplement connexe. On note $d$ le degr\'e de l'extension d\'eployant le tore $G^\tor=(G^\torf)^\circ$, $n$ l'ordre
du groupe fini $F=G/G^\circ$ et $m=nd$. Soit $H_0$ un $k$-sous-groupe fini de $G^\torf$ tel que la compos\'ee $H_0\to G^\torf\to F$ est surjective et soit $H$
le sous-groupe (fini aussi) de $G^\torf$ engendr\'e par $H_0$ et par $\phi_m^{-1}(H_0\cap G^\tor)$, o\`u $\phi_m$ d\'esigne la multiplication par $m$ dans
$G^\tor$. On suppose que $H$ est non ramifi\'e et que son ordre est premier \`a $p$.

En identifiant alors $\bral V$ avec $H^1(k,\hat G^\ab)$, le groupe de Brauer non ramifi\'e
alg\'ebrique $\brnral V$ de $V$ est donn\'e par les \'el\'ements $\alpha\in H^1(k,\hat G^\ab)$ v\'erifiant la propri\'et\'e suivante :

Pour toute extension finie non ramifi\'ee $k'$ de $k$, l'image $\alpha'$ de $\alpha$ dans $H^1(k',\hat G^\ab)$ est orthogonale au sous-ensemble
$\im[H^1(k',G)\to H^1(k',G^\ab)]$ pour l'accouplement induit par la dualit\'e locale.
\end{pro}

\begin{rem}
Les hypoth\`eses sur $H$ imposent que $F$ soit non ramifi\'e et d'ordre $n$ premier \`a $p$ et que $G^\tor$ soit d\'eploy\'e par une extension
$L/k$ de degr\'e $d$ premier \`a $p$. L'extension $L/k$ n'a pas besoin par contre d'\^etre non ramifi\'ee, car il est possible que $H\cap G^\tor$
soit d\'eploy\'e par une extension plus petite que $L(\zeta_n)$ et que celle-ci soit non ramifi\'ee alors que $L(\zeta_n)$ ne l'est pas. Il est important de
remarquer par ailleurs que ces hypoth\`eses sur $F$ et $G^\tor$, m\^eme en ajoutant le fait que $L/k$ soit non ramifi\'ee, ne suffisent pas a priori pour
assurer que $H$ soit non ramifi\'e. En d'autres mots, on ne peut pas obtenir un \'enonc\'e d\'ependant seulement de $G$ et de ses sous-quotients $F$
et $G^\tor$, sauf dans des cas tr\`es particuliers, comme par exemple lorsque $G=G^\tor\times F$.
\end{rem}

\begin{proof}
C'est une cons\'equence imm\'ediate du corollaire \ref{corollaire comparaison de Brnral entre G et H car 0}, de la proposition 4.23 de \cite{GLABrnral} et de la
proposition \ref{proposition Gille_Ch_R}. En effet, les deux premiers donnent le m\^eme \'enonc\'e avec $\im[H^1(k',H)\to H^1(k',H^\ab)]$ au lieu de
$\im[H^1(k',G)\to H^1(k',G^\ab)]$ et c'est la proposition \ref{proposition Gille_Ch_R} qui nous permet de remplacer le premier par le deuxi\`eme gr\^ace \`a
la compatibilit\'e des accouplements locaux.
\end{proof}

Enfin, ces deux derni\`eres propositions (\ref{proposition brnral se calcule en une extension finie} et \ref{proposition cohomologique corps local}) nous donnent
le r\'esultat suivant :

\begin{cor}
Sous les m\^emes hypoth\`eses que la proposition \ref{proposition cohomologique corps local}, \`a cela pr\`es qu'on suppose aussi que l'extension d\'eployant
$G^\tor$ est non ramifi\'ee, on a
\[\brnral V\subset H^1_\nr(k,\hat G^\ab).\]
\qed
\end{cor}

\section{Existence de sous-groupes finis}\label{section existence}
Soit maintenant $k$ un corps de caract\'eristique $p \geq 0$. On a vu que pour \'etudier le groupe de Brauer non ramifi\'e alg\'ebrique des espaces homog\`enes
$V=G\backslash G'$ pour un $k$-groupe $G$ il faut \'etudier, dans une extension
\[1\to T \to G\to F \to 1,\]
d'un groupe fini $F$ par un tore $T$, l'existence de sous-groupes finis $H$ se surjectant sur $F$. Dans ce sens, on a la proposition suivante.

\begin{pro}\label{proposition existence du sous-groupe fini}
Soit $k$ un corps de caract\'eristique $p\geq 0$. Soit $G$ un $k$-groupe alg\'ebrique lin\'eaire, extension d'un groupe fini $F$ d'ordre $n$ premier \`a $p$ par un tore $T$
d\'eploy\'e par une extension $L/k$ de degr\'e $d$, aussi premier \`a $p$. Alors il existe un $k$-sous-groupe fini $H$ de $G$ et un diagramme commutatif \`a
lignes exactes
\begin{equation}\label{equation diagramme existence sous-groupe fini}
\xymatrix{
1 \ar[r] & \asd{}{nd}{}{}{T} \ar[r] \ar@{^{(}->}[d] & H \ar[r] \ar@{^{(}->}[d] & F \ar[r] \ar@{=}[d] & 1 \\
1 \ar[r] & T \ar[r] & G \ar[r] & F \ar[r] & 1.
}
\end{equation}
\end{pro}

Avant de passer \`a la d\'emonstration de cette proposition, voyons en quelques cons\'equences imm\'ediates. La premi\`ere appara\^\i t en combinant ce r\'esultat
avec la proposition \ref{proposition Gille_Ch_R}.

\begin{cor}\label{corollaire sous-groupe fini surjectif en cohomologie}
Soit $k$ un corps de caract\'eristique $p\geq 0$. Soit $G$ un $k$-groupe alg\'ebrique lin\'eaire, extension d'un groupe fini $F$ d'ordre $n$ premier \`a
$p$ par un tore $T$ d\'eploy\'e par une extension $L/k$ de degr\'e $d$, aussi premier \`a $p$. Alors il existe un $k$-sous-groupe fini $H$ de $G$
(d'ordre $n^3d^2$) tel que le diagramme suivant est commutatif et \`a lignes exactes
\begin{equation}\label{equation diagramme sous-groupe fini pour formules}
\xymatrix{
1 \ar[r] & \asd{}{n^2d^2}{}{}{T} \ar[r] \ar@{^{(}->}[d] & H \ar[r] \ar@{^{(}->}[d] & F \ar[r] \ar@{=}[d] & 1 \\
1 \ar[r] & T \ar[r] & G \ar[r] & F \ar[r] & 1,
}
\end{equation}
et tel que, pour tout corps $k'\supset k$, l'application naturelle $H^1(k',H)\to H^1(k',G)$ est surjective.
\qed
\end{cor}

On peut ensuite pousser ce r\'esultat un peu plus loin, pour obtenir une g\'en\'eralisation de \cite[Cor. 1.4]{Gille_Ch_R}, donc aussi de \cite[Thm. 1.1(a)]{Gille_Ch_R_2}.

\begin{cor}\label{corollaire sous-groupe fini surjectif en cohomologie 2}
Soit $k$ un corps de caract\'eristique $p\geq 0$. Soit $G$ un $k$-groupe alg\'ebrique lin\'eaire lisse \`a composante connexe neutre $G^\circ$ r\'eductive. Soient $T$ un
$k$-tore maximal de $G^\circ$ et $W=N_G(T)/T$ le groupe de Weyl de $G$. Ce dernier est un $k$-groupe fini dont on note $w$ l'ordre. Soit enfin $L/k$ une extension d\'eployant
$T$ dont on note $d$ le degr\'e. On suppose que le produit $wd$ est premier \`a $p$.

Alors, il existe un $k$-sous-groupe fini $H$ de $G$ (d'ordre $w^3d^2$) tel que, pour tout corps $k'\supset k$, l'application naturelle
$H^1(k',H)\to H^1(k',G)$ est surjective.
\end{cor}

Ici, $N_G(T)$ d\'esigne le normalisateur de $T$ dans $G$. Rappelons par ailleurs que l'on a toujours le droit de consid\'erer un $k$-tore maximal $T$ de $G^\circ$ d'apr\`es
\cite[XIV, Cor. 3.20]{SGA3II}. On voit alors que, lorsque $k$ est de caract\'eristique nulle, les hypoth\`eses n'imposent aucune contrainte sur le groupe $G$ (\`a part le fait
d'\^etre \`a composante connexe neutre r\'eductive).

\begin{proof}
Notons $N=N_G(T)$. C'est donc une extension du groupe fini $W$ par le tore $T$. Le corollaire \ref{corollaire sous-groupe fini surjectif en cohomologie} assure alors
l'existence d'un sous-groupe $H$ d'ordre $w^3d^2$ tel que pour tout corps $k'\supset k$ l'application naturelle $H^1(k',H)\to H^1(k',N)$ est surjective.

Il suffit alors pour conclure de montrer que, pour tout corps $k'/k$, l'application $H^1(k',N)\to H^1(k',G)$ est surjective. Or cela est une cons\'equence facile de
\cite[III.\S2.2, Lem. 1]{SerreCohGal}. En effet, soit $\alpha\in H^1(k',G)$ et soit $a\in Z^1(k',G)$ un cocycle repr\'esentant $\alpha$. Le tordu $\asd{}{a}{}{}{G}$ de $G$
par le cocycle $a$ admet un $k'$-tore maximal dans sa composante connexe neutre (toujours d'apr\`es \cite[XIV, Cor. 3.20]{SGA3II}) et ce tore sera toujours conjugu\'e \`a $T$
sur la cl\^oture s\'eparable de $k'$ d'apr\`es \cite[Prop. A.2.10]{PseudoRedGps}. Le lemme cit\'e nous dit alors que $\alpha$ appartient \`a l'image de
$H^1(k',N)\to H^1(k',G)$, ce qui conclut.
\end{proof}

Enfin, on peut expliciter les \'enonc\'es des th\'eor\`emes \ref{theoreme cohomologique}, \ref{theoreme formule sur fq} et \ref{theoreme formule en car 0} en
enlevant l'hypoth\`ese d'existence d'un sous-groupe fini de $G^\tor$. On retrouve alors les r\'esultats suivants.

\begin{cor}\label{corollaire formule corps finis avec existence}
Soient $k$ un corps fini de cardinal $q$ et $G$ un $k$-groupe alg\'ebrique lin\'eaire \`a composante connexe neutre $G^\circ$ r\'eductive et tel que $F=G/G^\circ$ est d'ordre $n$
premier \`a $q$. Soit $V=G\backslash G'$ pour un plongement de $G$ dans $G'$ semi-simple simplement connexe. On suppose que le tore $G^\tor=(G^\torf)^\circ$
est d\'eploy\'e par une extension $L/k$ de degr\'e $d$ premier \`a $q$. Alors il existe un $k$-sous-groupe fini $H$ de $G^\torf$ tel que l'on a le diagramme
\eqref{equation diagramme sous-groupe fini pour formules} et tel que, en identifiant $\bral V$ avec $H^1(k,\hat G^\ab)$, le groupe de Brauer non ramifi\'e
alg\'ebrique $\brnral V$ de $V$ est donn\'e par les \'el\'ements $\alpha\in H^1(k,\hat G^\ab)$ v\'erifiant l'une des propri\'et\'es \'equivalentes suivantes :
\begin{itemize}
\item Pour toute extension finie $k'$ de $k$, l'image $\alpha'$ de $\alpha$ dans $H^1(k'((t)),\hat G^\ab)$ est orthogonale au sous-ensemble
$\im[H^1(k'((t)),H)\to H^1(k'((t)),H^\ab)]$ pour l'accouplement entre $H^1(k'((t)),\hat G^\ab)$ et $H^1(k'((t)),H^\ab)$ donn\'e par le cup produit et l'application
duale $\hat G^\ab\to \hat H^\ab$.
\item Pour $a\in Z^1(k,\hat G^\ab)$ un cocycle (quelconque) repr\'esentant $\alpha$, on a
\[a_{s}(N_{q,H}(b))=1\quad \forall\, b\in H(\bar k),\]
o\`u $N_{q,H}$ d\'esigne la $q$-norme sur $H$.
\end{itemize}
De plus, pour tout corps $k'\supset k$, l'application naturelle $H^1(k',H)\to H^1(k',G^\torf)$ est surjective.
\qed
\end{cor}

\begin{cor}\label{corollaire formule car 0 avec existence}
Soient $k$ un corps de caract\'eristique $0$, $G$ un $k$-groupe alg\'ebrique lin\'eaire et $V=G\backslash G'$ pour un plongement de $G$ dans $G'$ semi-simple
simplement connexe. Soit $L/k$ une extension d\'eployant le tore $G^\tor=(G^\torf)^\circ$. On note $d$ le degr\'e de cette extension et $n$ l'ordre du groupe
fini $F=G/G^\circ$. Alors il existe un $k$-sous-groupe fini $H$ de $G^\torf$ tel que l'on a le diagramme \eqref{equation diagramme sous-groupe fini pour formules}
et tel que, en identifiant $\bral V$ avec $H^1(k,\hat G^\ab)$, le groupe de Brauer non ramifi\'e alg\'ebrique $\brnral V$ de $V$ est donn\'e par les \'el\'ements
$\alpha\in H^1(k,\hat G^\ab)$ tels que, pour $a\in Z^1(k,\hat G^\ab)$ un cocycle (quelconque) repr\'esentant $\alpha$, on a
\[a_\sigma(N_{\sigma,H}(b))=1\quad \forall\, b\in H(\bar k),\forall\, \sigma\in\Gamma_k,\]
o\`u $N_{\sigma,H}$ d\'esigne la $\sigma$-norme sur $H$. De plus, pour tout corps $k'\supset k$, l'application naturelle $H^1(k',H)\to H^1(k',G^\torf)$ est surjective.
\qed
\end{cor}

Ces r\'esultats ne m\'eritant pas de d\'emonstration, on conclut cette section avec la preuve de la proposition \ref{proposition existence du sous-groupe fini}.

\begin{proof}[D\'emonstration de la proposition \ref{proposition existence du sous-groupe fini}]
De la suite exacte de $k$-groupes
\[1\to T\to G\to F\to 1,\]
on d\'eduit la suite exacte de $\Gamma_k$-groupes
\[1\to T(\bar k)\to G(\bar k) \to F(\bar k)\to 1,\]
et puisque $T(\bar k)$ est commutatif, on a une action de $F(\bar k)$ sur $T(\bar k)$ par conjugaison. Or, l'\'equivariance de cette action par rapport aux actions
de $\Gamma_k$ sur $F(\bar k)$ et $T(\bar k)$ induit une action continue du groupe profini $F(\bar k)\rtimes\Gamma_k$ sur $T(\bar k)$. En effet, tout
\'el\'ement $h\in F(\bar k)\rtimes\Gamma_k$ s'\'ecrit de fa\c con unique de la forme $h=\sigma_h f_h$ avec $\sigma_h\in\Gamma_k$ et $f_h\in F(\bar k)$, d'o\`u
l'on peut d\'efinir l'action de $F(\bar k)\rtimes\Gamma_k$ sur $T(\bar k)$ tout simplement en posant
\[h\cdot t:=\asd{\sigma_h}{}{}{}{(}f_h\cdot t),\]
pour $h\in F(\bar k)\rtimes\Gamma_k$ et $t\in T(\bar k)$. L'\'equivariance des actions de $F(\bar k)$ et $\Gamma_k$, exprim\'ee par la formule
\[\asd{\sigma}{}{}{}{(}f\cdot t)=\asd{\sigma}{}{}{}{f}\cdot\asd{\sigma}{}{}{}{t},\quad\forall\,f\in F(\bar k),\,\sigma\in\Gamma_k,\,t\in T(\bar k),\]
nous permet alors de d\'eduire (apr\`es un petit calcul) que l'action de $F(\bar k)\rtimes\Gamma_k$ est bien d\'efinie.\\

Consid\'erons alors le groupe $G(\bar k)\rtimes\Gamma_k$ par rapport \`a l'action naturelle de $\Gamma_k$ sur $G(\bar k)$. Puisque
$T(\bar k)$ est stable par cette action, on voit qu'il est un sous-groupe distingu\'e de $G(\bar k)\rtimes\Gamma_k$, ce qui nous donne l'extension de groupes
\[1\to T(\bar k)\to G(\bar k)\rtimes\Gamma_k \to F(\bar k)\rtimes\Gamma_k\to 1,\]
o\`u l'action de $F(\bar k)\rtimes\Gamma_k$ sur $T(\bar k)$ par conjugaison correspond pr\'ecis\'ement \`a celle d\'ecrite ci-dessus. Cette extension repr\'esente
une classe $\alpha\in H^2(F(\bar k)\rtimes\Gamma_k,T(\bar k))$.\footnote{On remarquera que, $F(\bar k)\rtimes\Gamma_k$ \'etant un groupe profini, on traite alors
$T(\bar k)$ comme un $(F(\bar k)\rtimes\Gamma_k)$-module discret. Le groupe $G(\bar k)\rtimes\Gamma_k$ a alors comme base de voisinages de l'identit\'e les sous-groupes
de la forme $\{1\}\rtimes \Delta$ avec $\Delta$ ouvert dans $\Gamma_k$. Il est alors facile de voir aussi que dans la suite exacte la fl\`eche de gauche est
continue et ferm\'ee et celle de droite est continue et ouverte, et ce sont pr\'ecis\'ement ces extensions qui sont classifi\'ees par le groupe
$H^2(F(\bar k)\rtimes\Gamma_k,T(\bar k))$, cf. \cite[Thm. 1.2.4]{NSW}.}
Or, on sait que la composition des fl\`eches
\[H^2(F(\bar k)\rtimes\Gamma_k,T(\bar k))\xrightarrow{\res} H^2(\Gamma_k,T(\bar k))\xrightarrow{\cores} H^2(F(\bar k)\rtimes\Gamma_k,T(\bar k)),\]
correspond \`a la multiplication par $n$ dans $H^2(F(\bar k)\rtimes\Gamma_k,T(\bar k))$. D'autre part, on voit facilement que la restriction de $\alpha$ \`a
$H^2(\Gamma_k,T(\bar k))$ correspond \`a l'extension triviale $T(\bar k)\rtimes \Gamma_k$ (on rappelle que l'extension repr\'esentant $\alpha$ a \'et\'e construite
\`a partir du groupe $G(\bar k)\rtimes\Gamma_k$), ce qui nous dit que $n\alpha=0$ dans $H^2(F(\bar k)\rtimes\Gamma_k,T(\bar k))$. Si l'on consid\`ere alors la
suite exacte de $F(\bar k)\rtimes\Gamma_k$-groupes
\[1\to \asd{}{n}{}{}{T}(\bar k)\to T(\bar k) \xrightarrow{\cdot n} T(\bar k) \to 1,\]
et la suite exacte longue de cohomologie qui lui est associ\'ee,
\[H^2(F(\bar k)\rtimes\Gamma_k,\asd{}{n}{}{}{T}(\bar k))\to H^2(F(\bar k)\rtimes\Gamma_k,T(\bar k))\xrightarrow{\cdot n}H^2(F(\bar k)\rtimes\Gamma_k,T(\bar k)),\]
on voit que la classe $\alpha$ provient de $H^2(F(\bar k)\rtimes\Gamma_k,\asd{}{n}{}{}{T}(\bar k))$.

Maintenant, consid\'erons le diagramme commutatif \`a lignes exactes suivant :
\[\xymatrix{
1 \ar[r] & \asd{}{nd}{}{}{T}(\bar k) \ar[r]^i & T(\bar k) \ar[r]^{\cdot nd} & T(\bar k) \ar[r] & 1\\
1 \ar[r] & \asd{}{n}{}{}{T}(\bar k) \ar[r]^i \ar[u]^i & T(\bar k) \ar@{=}[u] \ar[r]^{\cdot n} & T(\bar k) \ar[u]^{\cdot d} \ar[r] & 1,
}\]
o\`u $i$ d\'esigne les inclusions \'evidentes. En consid\'erant en m\^eme temps ces groupes comme des $(F(\bar k)\rtimes\Gamma_k)$-modules et par restriction comme
des $\Gamma_k$-modules, on obtient un diagramme commutatif de suites exactes longues de cohomologie
\[\xymatrix @C=-15mm @R=5mm{
& H^1(F(\bar k)\rtimes\Gamma_k,T(\bar k)) \ar'[d][dd] \ar[rr] & & H^2(F(\bar k)\rtimes\Gamma_k,\asd{}{nd}{}{}{T}(\bar k)) \ar'[d][dd] \ar[rr] & &
H^2(F(\bar k)\rtimes\Gamma_k,T(\bar k)) \ar[dd] \\
H^1(F(\bar k)\rtimes\Gamma_k,T(\bar k)) \ar[dd] \ar[rr] \ar[ur]^{\cdot d} & & H^2(F(\bar k)\rtimes\Gamma_k,\asd{}{n}{}{}{T}(\bar k)) \ar[dd] \ar[rr] \ar[ur] & &
H^2(F(\bar k)\rtimes\Gamma_k,T(\bar k)) \ar[dd] \ar@{=}[ur] \\
& H^1(\Gamma_k,T(\bar k)) \ar'[r][rr] & & H^2(\Gamma_k,\asd{}{nd}{}{}{T}(\bar k)) \ar'[r][rr] & & H^2(\Gamma_k,T(\bar k)) \\
H^1(\Gamma_k,T(\bar k)) \ar[rr] \ar[ur]^{\cdot d} & & H^2(\Gamma_k,\asd{}{n}{}{}{T}(\bar k)) \ar[rr] \ar[ur] & & H^2(\Gamma_k,T(\bar k)), \ar@{=}[ur] \\
}\]
o\`u toutes les fl\`eches verticales sont des applications de restriction. \`A partir de ce qui a d\'ej\`a \'et\'e dit, on fait la chasse au diagramme suivante.
La classe $\alpha\in H^2(F(\bar k)\rtimes\Gamma_k,T(\bar k))$ dont on est parti provient d'une classe
$\alpha_n\in H^2(F(\bar k)\rtimes\Gamma_k,\asd{}{n}{}{}{T}(\bar k))$, dont l'image $\beta_n\in H^2(\Gamma_k,\asd{}{n}{}{}{T}(\bar k))$ est envoy\'e en $0$ dans
$H^2(\Gamma_k,T(\bar k))$, d'o\`u l'on sait qu'elle provient d'une classe $\gamma\in H^1(\Gamma_k,T(\bar k))$. Or, puisque $T$ est d\'eploy\'e par $L$, on sait que
la restriction de $\gamma$ \`a $H^1(\Gamma_L,T(\bar k))$ est triviale et l'argument classique de restriction-corestriction nous dit alors que $\gamma$ est de
$d$-torsion. On voit alors que l'image de $\gamma$ apr\`es multiplication par $d$ est $0$, ce qui nous dit que l'image
$\beta_{nd}\in H^2(\Gamma_k,\asd{}{nd}{}{}{T}(\bar k))$ de $\beta_n$ est triviale. Si l'on note alors
$\alpha_{nd}\in H^2(F(\bar k)\rtimes\Gamma_k,\asd{}{nd}{}{}{T}(\bar k))$ l'image de $\alpha_n$, on voit qu'elle rel\`eve aussi $\alpha$ et que sa restriction
\`a $H^2(\Gamma_k,\asd{}{nd}{}{}{T}(\bar k))$ est triviale. Ceci se r\'esume dans le diagramme suivant :
\[\xymatrix@R=3mm{
& & & \alpha_{nd} \ar@{|->}[rr] \ar@{|->}'[d][dd] & & \alpha \ar@{|->}[dd] \\
& & \alpha_n \ar@{|->}[rr] \ar@{|->}[dd] \ar@{|->}[ur] & & \alpha \ar@{|->}[ur] \ar@{|->}[dd] \\
& d\gamma=0 \ar@{|->}'[r][rr] & & \beta_{nd}=0 \ar@{|->}'[r][rr] & & 0 \\
\gamma \ar@{|->}[rr] \ar@{|->}[ur] & & \beta_n \ar@{|->}[ur] \ar@{|->}[rr] & & 0. \ar@{|->}[ur]
}\]
Consid\'erons le carr\'e arri\`ere de droite dans le diagramme. Puisque ces classes correspondent \`a des extensions de groupes, si l'on note $\bar E$ une
extension repr\'esentant $\alpha_{nd}$ on trouve le diagramme commutatif \`a lignes exactes suivant :
\[\xymatrix@C=5mm @R=3mm{
& \alpha_{nd} : \ar@{|->}[dl] \ar@{|->}[dd] & & \asd{}{nd}{}{}{T}(\bar k) \ar@{^{(}->}[rr] \ar'[d][dd] & & \bar E \ar@{->>}[rr] \ar'[d][dd]
& & F(\bar k)\rtimes\Gamma_k \ar[dd]\\
\beta_{nd} : \ar@{|->}[dd] & & \asd{}{nd}{}{}{T}(\bar k) \ar@{^{(}->}[rr] \ar[dd] \ar[ur] & & \asd{}{nd}{}{}{T}(\bar k)\rtimes\Gamma_k \ar@{->>}[rr] \ar[dd] \ar[ur]
& & \Gamma_k \ar[dd] \ar[ur] \\
& \alpha : \ar@{|->}[dl] & & T(\bar k) \ar@{^{(}->}'[r][rr] & & G(\bar k)\rtimes\Gamma_k \ar@{->>}'[r][rr] & & F(\bar k)\rtimes\Gamma_k \\
0 : & & T(\bar k) \ar@{^{(}->}[rr] \ar[ur] & & T(\bar k)\rtimes\Gamma_k \ar@{->>}[rr] \ar[ur] & & \Gamma_k \ar[ur],
}\]
o\`u les deux extensions qui sont devant dans le diagramme sont des produits semi-directs car elles correspondent aux classes triviales dans le diagramme pr\'ec\'edent.
Remarquons enfin que, en notant $\bar H$ le groupe abstrait d\'efini comme la pr\'eimage dans $\bar E$ du sous-groupe $F(\bar k)$ de $F(\bar k)\rtimes\Gamma_k$, on
obtient l'extension
\[1\to\bar H\to \bar E\to\Gamma_k\to 1.\]
Si l'on consid\`ere alors la section naturelle $\Gamma_k\to \asd{}{nd}{}{}{T}(\bar k)\rtimes\Gamma_k$ et qu'on la compose avec la fl\`eche
$\asd{}{nd}{}{}{T}(\bar k)\rtimes\Gamma_k\to \bar E$ du diagramme, on trouve une section $s':\Gamma_k\to\bar E$ compatible avec la section naturelle
$s:\Gamma_k\to G(\bar k)\rtimes\Gamma_k$, i.e. le diagramme
\[\xymatrix{
\bar H \ar[r] \ar[d] & \bar E \ar[r] \ar[d] & \Gamma _k \ar@/_1.5pc/[l]_{s'} \ar@{=}[d] \\
G(\bar k) \ar[r] & G(\bar k)\rtimes\Gamma_k \ar[r] & \Gamma_k, \ar@/_1.5pc/[l]_{s}
}\]
est commutatif. Ceci nous dit que l'action naturelle de $\Gamma_k$ sur $G(\bar k)$ laisse $\bar H$ invariant. De plus
on peut toujours voir $\bar H$ comme un sous-$\bar k$-groupe de $G_{\bar k}$. En effet, en tant que r\'eunion finie de points ferm\'es, $\bar H$ est une
sous-$\bar k$-vari\'et\'e de $G_{\bar k}$ qui v\'erifie clairement tous les diagrammes commutatifs d\'efinissant un $\bar k$-groupe alg\'ebrique et un morphisme
de $\bar k$-groupes. L'existence du diagramme \eqref{equation diagramme existence sous-groupe fini} est alors \'evidente au niveau des $\bar k$-points.
Le lemme suivant nous permet alors de conclure que $\bar H$ descend en un $k$-sous-groupe de $G$, ce qui conclut la preuve.
\end{proof}

\begin{lem}
Soient $k$ un corps de caract\'eristique $p\geq 0$, $G$ un $k$-groupe alg\'ebrique lisse et $\bar H$ un $\bar k$-sous-groupe fini de $G_{\bar k}$ d'ordre premier
\`a $p$. Supposons que l'action naturelle du groupe de Galois $\Gamma_k$ sur $G(\bar k)$ laisse $\bar H(\bar k)$ invariant. Alors il existe une $k$-forme $H$ de
$\bar H$ et un $k$-morphisme $H\to G$ induisant l'inclusion de $\bar H$ dans $G_{\bar k}$.
\end{lem}

\begin{proof}
La structure de $k$-groupe de $G$ nous donne une section $s$ de la suite exacte
\[\xymatrix{
1 \ar[r] & \aut_{\bar k} (G_{\bar k}) \ar[r] & \saut (G_{\bar k}) \ar[r] & \Gamma_k \ar@/_1.5pc/[l]_{s}.
}\]
o\`u $\saut (G_{\bar k})$ d\'esigne le groupe d'automorphismes semi-alg\'ebriques de $G_{\bar k}$ (cf. par exemple \cite[\S1]{FSS} pour la d\'efinition de
$\saut (G_{\bar k})$ et ses propri\'et\'es, dont notamment cette suite exacte).
Or, l'action de $\Gamma_k$ sur $G(\bar k)$ se d\'efinit via le morphisme naturel $\pi:\saut (G)\to \aut (G(\bar k))$. On voit alors que l'on a un morphisme
\[\phi:\Gamma_k\xrightarrow{\pi\circ s}\aut(G(\bar k))\to\aut(H(\bar k)).\]
Si l'on rappelle maintenant que, puisque $\bar H$ est \emph{fini}, on a $\aut(H(\bar k))=\aut_{\bar k}(\bar H)$ et
$\saut (\bar H)\cong\aut_{\bar k}(\bar H)\times\Gamma_k$ (cela d\'ecoule du fait que l'on a une $k$-forme canonique de $\bar H$ qui est le $k$-groupe constant
associ\'e au groupe fini abstrait $\bar H(\bar k)$), on voit que l'on obtient une section $s'$ de la suite exacte
\[\xymatrix{
1 \ar[r] & \aut_{\bar k} (\bar H) \ar[r] & \saut (\bar H) \ar[r] & \Gamma_k \ar@/_1.5pc/[l]_{s'},
}\]
tout simplement en posant $s'=\phi\times\mathrm{id}$. Cette section correspond \`a la donn\'ee d'une $k$-forme $H$ de $\bar H$ dont l'action de $\Gamma_k$ induite
sur les $\bar k$-points correspond pr\'ecis\'ement \`a celle induite par restriction de l'action de $\Gamma_k$ sur $G(\bar k)$ car toutes les deux sont
d\'efinies via $\phi$. Il est alors \'evident que pour tout $\sigma\in \Gamma_k$, l'automorphisme $\sigma$-semi-alg\'ebrique d\'efini par $\sigma$ sur $G_{\bar k}$
induit bien l'automorphisme $\sigma$-semi-alg\'ebrique sur $\bar H$ donn\'e par $s'(\sigma)$ et qu'alors l'inclusion $H_{\bar k}\hookrightarrow G_{\bar k}$
provient bien d'un $k$-morphisme de groupes alg\'ebriques.
\end{proof}

\begin{rem}
Lorsque le corps de base est parfait et de dimension cohomologique $\leq 1$, comme c'est le cas des corps finis et des corps quasi-finis en caract\'eristique $0$, on
peut m\^eme d\'emontrer que le sous-groupe fini $H$ peut \^etre obtenu en prenant seulement la $n$-torsion du tore $T$. En effet, on peut toujours trouver une extension
\[1\to \asd{}{n}{}{}{T}(\bar k)\to \bar E \to F(\bar k)\rtimes\Gamma_k\to 1,\]
avec une inclusion compatible dans $G(\bar k)\rtimes\Gamma_k$. En consid\'erant alors la pr\'eimage $\bar H$ de $F(\bar k)$ dans $\bar E$, cela nous donne l'extension
\[1\to \bar H\to \bar E\to\Gamma_k\to 1,\]
ce qui nous dit que $\Gamma_k$ agit ext\'erieurement sur $\bar H$. Or, puisque $k$ est parfait et de dimension cohomologique $\leq 1$, le th\'eor\`eme
dans \cite[3.5]{SpringerH2} nous dit que cette action ext\'erieure provient d'une structure de $k$-groupe. Cette structure n'est pas n\'ecessairement compatible avec
celle de $G$, mais apr\`es une chasse au diagramme et un argument de torsion on peut d\'emontrer que, en tordant $G$ par un cocycle dont la classe dans $H^1(k,G)$
est \emph{triviale}, l'inclusion de $\bar H$ dans $G_{\bar k}$ descend en un $k$-morphisme de groupes. Ensuite, en revenant en arri\`ere par la torsion,
l'image de ce $k$-sous-groupe du tordu correspond au sous-groupe fini voulu.
\end{rem}

\section{Application : obstruction de Brauer-Manin \`a l'approximation faible}\label{section BM}
Pour conclure ce texte, on donne une application de cette formule \`a l'obstruction de Brauer-Manin pour l'approximation faible correspondant
\`a l'analogue de ce qui a \'et\'e fait dans \cite[\S 5.2]{GLABrnral}. Il s'agit de montrer que l'obstruction associ\'ee au groupe $\brnral V$ ne prend pas en
compte ce qui se passe aux places r\'eelles, ce qu'on fait en profitant du fait que sur $\bb{R}$ la formule devient assez simple \`a raison de la
petitesse de son groupe de Galois absolu.

\begin{pro}\label{proposition BM ne voit pas les places reelles}
Soient $k=\bb{R}$, $G$ un $k$-groupe alg\'ebrique lin\'eaire plong\'e dans $G'$ semi-simple simplement connexe et $V=G\backslash G'$. En identifiant $\bral V$ avec
$H^1(k,\hat G^\ab)$, tout \'el\'ement $\alpha\in\brnral V$ est orthogonal au sous-ensemble $\im [H^1(k,G)\to H^1(k,G^\ab)]$ de $H^1(k,G^\ab)$.
\end{pro}

\begin{rem}
Ce r\'esultat nous dit en particulier que si l'obstruction de Brauer-Manin associ\'ee au sous-groupe $\brnral$ s'av\`ere \^etre la seule pour $V$,
alors $V$ v\'erifie l'approximation r\'eelle, comme il sera explicit\'e dans le corollaire \ref{corollaire approximation reelle} ci-dessous. Pris dans
l'autre sens, ceci veut dire aussi qu'un espace homog\`ene comme ci-dessus ne v\'erifiant pas l'approximation r\'eelle fournirait un exemple de
vari\'et\'e o\`u l'obstruction de Brauer-Manin n'est pas la seule (supposant qu'il n'y ait pas d'obstructions g\'eom\'etriques).
\end{rem}

\begin{proof}
On a le diagramme commutatif
\[\xymatrix{
G \ar@{->>}[r] \ar@{->>}[d] & G^\ab \ar@{->>}[d] \\
G^{\torf} \ar@{->>}[r] & (G^{\torf})^\ab,
}\]
o\`u l'on remarque que le noyau de la fl\`eche verticale de droite est un groupe ab\'elien unipotent, d'o\`u l'isomorphisme
\[H^1(k,G^\ab)\xrightarrow{\sim}H^1(k,(G^\torf)^\ab).\]
On en d\'eduit qu'il suffit de montrer que tout \'el\'ement $\alpha\in\brnral V\subset H^1(k,\hat G^\ab)$ est orthogonal au sous-ensemble
$\im [H^1(k,G^\torf)\to H^1(k,(G^\torf)^\ab)]$ de $H^1(k,(G^\torf)^\ab)$ (on remarque que $(G^\torf)^\ab$ est le dual de $\hat G^\ab$ et donc cette
orthogonalit\'e a bien un sens).

Soit $\sigma$ l'\'el\'ement non nul de $\Gamma_k$ et soit $H$ le sous-groupe fini de $G^\torf$ donn\'e par le corollaire
\ref{corollaire formule car 0 avec existence}. Il est facile de voir que l'on a les isomorphismes
\begin{align*}
Z^1(k,H) \xrightarrow{\sim} H^{\varphi_\sigma}, &\quad b\mapsto b_\sigma \\
Z^1(k,H^\ab)\xrightarrow{\sim}  (H^\ab)^{\varphi_\sigma}, &\quad \b\mapsto \b_\sigma.
\end{align*}
Soit $g\in H^{\varphi_\sigma}$. Puisque $\varphi_\sigma(g)=g$, on a que $N_\sigma(g)=\bar g\in H^\ab$. En particulier, on voit que l'image de
$N_\sigma$ contient l'image de $Z^1(k,H)$ dans $Z^1(k,H^\ab)$.

Soit alors $\beta\in H^1(k,(G^\torf)^\ab)$ un \'el\'ement dans $\im [H^1(k,G^\torf)\to H^1(k,(G^\torf)^\ab)]$. La surjectivit\'e de $H^1(k,H)\to H^1(k,G^\torf)$ nous dit que
l'on peut repr\'esenter $\beta$ par un cocycle $\b$ \`a valeurs dans $H^\ab$ et tel que $\b_\sigma\in N_\sigma(H)$. Pour $\alpha\in\brnral V$, soit
$a\in Z^1(k,\hat G^\ab)$ un cocycle repr\'esentant $\alpha$. On sait alors d'apr\`es le corollaire \ref{corollaire formule car 0 avec existence} que l'on a
$a_\sigma(\b_\sigma)=1$. D'autre part, on sait que $\alpha\cup\beta\in H^2(k,\mu_{n^3d^2})\subset\br k$ est repr\'esent\'e par un cocycle $c$ tel que
\[c_{\sigma,\sigma}=a_\sigma(\asd{\sigma}{}{}{\sigma}{\b})=a_{\sigma}(\b_\sigma^{-1})=(a_\sigma(\b_\sigma))^{-1}=1,\]
car on rappelle que $\b_\sigma\in (H^\ab)^{\varphi_\sigma}$ et alors $\asd{\sigma}{}{}{\sigma}{\b}=\b_\sigma^{-1}$. Il est clair alors que
$\alpha\cup\beta$ est nul, ce qui nous dit que $\alpha$ est orthogonal au sous-ensemble $\im [H^1(k,G^\torf)\to H^1(k,(G^\torf)^\ab)]$ de $H^1(k,(G^\torf)^\ab)$.
\end{proof}

Soit maintenant $k$ un corps de nombres et $V=G\backslash G'$ comme toujours. Ce dernier r\'esultat et la proposition \ref{proposition cohomologique corps local}
nous donnent le r\'esultat annonc\'e, sous la forme plus g\'en\'erale suivante.

\begin{cor}\label{corollaire Brauer-Manin}
Soient $k$ un corps de nombres, $G$ un $k$-groupe alg\'ebrique lin\'eaire plong\'e dans $G'$ semi-simple simplement connexe et
$V=G\backslash G'$. Soit $H$ le $k$-sous-groupe fini
de $G^\torf$ donn\'e par le corollaire \ref{corollaire formule car 0 avec existence}. Soit $L/k$ une extension d\'eployant $H$. Soit
$S$ la r\'eunion de l'ensemble des places non archim\'ediennes ramifi\'ees pour l'extension $L/k$ avec l'ensemble des places (non archim\'ediennes)
se trouvant au-dessus d'un premier qui divise l'ordre du groupe $H$. Soit enfin $T=\Omega_k\smallsetminus S$. Alors la projection
\[\left(\prod_{\Omega_k} V(k_v)\right)^{\brnral}\to \prod_{v\in T} V(k_v),\]
est surjective.
\end{cor}

On rappelle que l'ensemble $(\prod_{\Omega_k} V(k_v))^{\brnral}$ correspond aux \'el\'ements de  $\prod_{\Omega_k} V(k_v)$ qui sont
orthgonaux au groupe $\brnral V$ pour l'accouplement de Brauer-Manin. Pour la d\'efinition de cet accouplement, ainsi qu'une d\'efinition
g\'en\'erale de l'obstruction de Brauer-Manin, on renvoie \`a \cite[\S5.2]{Skor}.

\begin{rem}
Si l'on note $d$ le degr\'e de l'extension d\'eployant $G^\tor$ et $n$ le cardinal de $F$, alors les premiers divisant l'ordre de $H$ sont
ceux qui divisent le produit $nd$ puisque le groupe $H$ donn\'e par le corollaire \ref{corollaire formule car 0 avec existence} est en fin du compte
celui du corollaire \ref{corollaire sous-groupe fini surjectif en cohomologie}. L'ensemble $S$ est alors compos\'e des places non archim\'ediennes
ramifi\'ees pour $L/k$ et de celles au dessus d'un premier divisant $nd$. Par ailleurs, si l'on est oblig\'es d'introduire le sous-groupe fini $H$ dans
l'\'enonc\'e, c'est parce qu'on n'a aucun contr\^ole sur l'extension d\'eployant $H$. En effet, si l'on note $k_F$ et $k_\tor$ les extensions d\'eployant
$F$ et la $n^2d^2$-torsion de $G^\tor$, rien ne nous dit que la compos\'ee de $k_F$ et $k_\tor$ d\'eploie l'extension $H$ de ces deux groupes finis.
\end{rem}

Enfin, si l'on note $\Omega_\infty$ l'ensemble des places archim\'ediennes de $k$ et l'on remarque que $\Omega_\infty\subset T$, on obtient en particulier :

\begin{cor}\label{corollaire approximation reelle}
Sous les hypoth\`eses du corollaire \ref{corollaire Brauer-Manin}, si l'obstruction de Brauer-Manin associ\'ee au sous-groupe $\brnral$ est la seule pour $V$,
alors $V$ v\'erifie l'approximation r\'eelle. En d'autres mots, si $V(k)$ est dense dans $(\prod_{\Omega_k} V(k_v))^{\brnral}$, alors $V(k)$ est
dense dans $\prod_{\Omega_\infty}V(k_v)$, o\`u $\Omega_\infty$ d\'esigne l'ensemble des places archim\'ediennes de $k$.\qed
\end{cor}

\begin{proof}[D\'emonstration du corollaire \ref{corollaire Brauer-Manin}]
Il suffit de noter que d'apr\`es le d\'ebut de \cite[\S 4]{GLABrnral}, l'ensemble $(\prod_{\Omega_k}V(k_v))^{\brnral}$ est l'ensemble de familles
de points $(P_v)_{\Omega_k}$ telles que
\[\sum_{v\in\Omega_k}\mathrm{inv}_v(\alpha_v\cup [Z](P_v))=0\in\q/\z,\quad\forall \alpha\in\brnral V\subset H^1(k,\hat G^\ab),\]
o\`u $\mathrm{inv}_v$ est l'application canonique $\br k_v\to\q/\z$ donn\'ee par la th\'eorie du corps de classes et $[Z](P_v)\in H^1(k,G^\ab)$ est la classe
du $k_v$-torseur sous $G^\ab$ obtenu \`a partir du torseur $Z=D(G)\backslash G'\to V$ sous $G^\ab$ et du $k_v$-point $P_v\in V(k_v)$. 

En effet, on voit alors imm\'ediatement \`a partir des propositions \ref{proposition BM ne voit pas les places reelles} et
\ref{proposition cohomologique corps local} qu'on a $\alpha_v\cup [Z](P_v)=0$ pour toute place $v\in T$. Pour les places archim\'ediennes c'est \'evident d'apr\`es la
premi\`ere de ces propositions, tandis que pour toute telle place non archim\'edienne on a que $H$ est non ramifi\'e en tant que $k_v$-groupe et son cardinal est premier \`a
la caract\'eristique r\'esiduelle, ce qui nous donne le droit d'utiliser la deuxi\`eme. Il est alors \'evident que pour $(\beta_v)_T\in\prod_{v\in T} V(k_v)$,
l'\'el\'ement $(0)_S\times (\beta_v)_T\in\prod_{\Omega_k}V(k_v)$ appartient \`a $(\prod_{\Omega_k}V(k_v))^{\brnral}$, ce qui conclut.
\end{proof}

\appendix

\section{Une remarque de Colliot-Th\'el\`ene}
Pendant la r\'edaction de ce texte, Colliot-Th\'el\`ene nous a fait remarquer que l'on peut d\'emontrer le r\'esultat suivant (cf. \cite{ColliotBrnr}) :

\begin{pro}[Colliot-Th\'el\`ene]
Soient $k$ un corps de caract\'eristique $0$, $G$ un $k$-groupe alg\'ebrique lin\'eaire non n\'ecessairement connexe et $r_i:G\hookrightarrow G'_i$ des plongements de $G$ dans
des $k$-groupes semi-simples simplement connexes. Alors
\[\brnr(G\backslash G'_1)\cong\brnr(G\backslash G'_2).\]
\end{pro}

La d\'emonstration de ce r\'esultat utilise notamment le th\'eor\`eme 4.2 de \cite{ColliotKunyavskii2}, mais elle ne profite pas de toute sa puissance. On modifie
alors un peu ici le raisonnement de Colliot-Th\'el\`ene (en suivant pourtant les m\^emes lignes) pour d\'emontrer le r\'esultat suivant :

\begin{pro}\label{proposition remarque Colliot}
Soient $k$ un corps de caract\'eristique $0$, $G$ un $k$-groupe alg\'ebrique lin\'eaire non n\'ecessairement connexe et $r_i:G\to G'_i$ des morphismes de groupes alg\'ebriques,
\`a noyau simplement connexe, dans des $k$-groupes semi-simples simplement connexes. Alors
\[\brnr(r_1(G)\backslash G'_1)\cong\brnr(r_2(G)\backslash G'_2).\]
\end{pro}

On rappelle qu'un $k$-groupe alg\'ebrique lin\'eaire $H$ est simplement connexe si et seulement s'il est une extension d'un groupe semi-simple simplement connexe par un
groupe unipotent. On pourrait demander de fa\c con \'equivalente que $\bar k[H]^*=\bar k^*$ et que $\pic H_{\bar k}=0$.\\

On commence donc la d\'emonstration avec le lemme suivant :

\begin{lem}\label{lemme Colliot}
Soit $k$ un corps de caract\'eristique $0$. Soient $G'$ un $k$-groupe semi-simple simplement connexe, $V$ un espace homog\`ene de $G'$ \`a stabilisateur
g\'eom\'etrique $\bar H$ simplement connexe et $X$ une $k$-vari\'et\'e projective lisse g\'eom\'etriquement connexe $k$-birationnelle \`a $V$. Alors
les fl\`eches naturelles
\[\br k\to\br V\quad\text{et}\quad\br k\to\br X,\]
sont des isomorphismes.
\end{lem}

\begin{proof}
D'apr\`es ce qui a \'et\'e dit dans les rappels de la section \ref{section notations et rappels}, on a $\bar k[V]^*=\bar k^*$ et $\pic V_{\bar k}=\hat{\bar H}=0$,
d'o\`u la suite exacte donn\'ee par \cite[Lem. 6.3 (i)]{Sansuc81} nous dit que $\br k\xrightarrow{\sim}\brun V$. Il suffit alors pour avoir
$\br k\xrightarrow{\sim} \br V$ de d\'emontrer que $\br V_{\bar k}=0$, ce qui est \'evident d'apr\`es \cite[Prop. 6.10]{Sansuc81} et le fait que
$\pic \bar H =0$ et $\br G'_{\bar k}=0$ car $G'$ est semi-simple simplement connexe.

Maintenant, il existe un ouvert $U$ non vide de $V$ et un $k$-morphisme birationnel $U\to X$. Puisque $X$ est projective et $V$ lisse, on peut supposer que $U$
contient tous les points de codimension 1 de $V$, donc par puret\'e du groupe de Brauer (cf. \cite[Cor. 3.4.2]{ColliotSantaBarbara}), la
restriction $\br V\to\br U$ et par suite la fl\`eche $\br k\to \br U$ sont des isomorphismes. D'autre part, on a le diagramme commutatif suivant 
\[\xymatrix{
\br k \ar[r] \ar[d]_{\sim} & \br X \ar@{^{(}->}[d] \ar[dl] \\
\br U \ar@{^{(}->}[r] & \br (k(X)),
}\] 
o\`u les fl\`eches $\hookrightarrow$ d\'esignent des applications injectives (cf. \cite[Thm. 5.11]{ColliotSansucChili}). Il est facile alors d'en d\'eduire que
l'on a $\br k\xrightarrow{\sim}\br X$.
\end{proof}

Avec ce lemme on peut montrer la proposition suivante, o\`u intervient enfin le th\'eor\`eme 4.2 de \cite{ColliotKunyavskii2} comme on l'avait annonc\'e.

\begin{pro}\label{proposition Colliot}
Soit $k$ un corps de caract\'eristique 0. Soient $X$ et $Y$ deux $k$-vari\'et\'es projectives, lisses, g\'eom\'etriquement connexes et $f:X\to Y$ un morphisme
dominant dont la fibre g\'en\'erique $X_\eta/k(Y)$ est $k(Y)$-birationnelle \`a un espace homog\`ene d'un $k(Y)$-groupe $G'$ semi-simple simplement connexe
\`a stabilisateur simplement connexe. Alors :
\begin{enumerate}
\item En tout point y de codimension $1$ de $Y$ la fibre $X_y/k(y)$ contient une composante connexe g\'eom\'etriquement int\`egre de multiplicit\'e 1.
\item L'application $f^*:\br Y\to\br X$ est un isomorphisme.
\end{enumerate}
\end{pro}

\begin{proof}
Dans le cas o\`u $X$ contient un ouvert $k(Y)$-isomorphe \`a un espace homog\`ene de $G'$ \`a stabilisateur simplement connexe, le premier point
de l'\'enonc\'e est une cons\'equence du th\'eor\`eme 4.2 de \cite{ColliotKunyavskii2}. En effet, il suffit de consid\`erer l'anneau de valuation discr\`ete $A_y$
de corps de fractions $k(Y)$ correspondant \`a $y$ (son spectre c'est la r\'eunion de $\eta$ et $y$ en tant que sch\'ema), puis de consid\`erer la fibre $X_{A_y}$
sur $A_y$. La proposition 3.9(b) de \cite{ColliotCIME} nous permet alors de d\'emontrer le cas g\'en\'eral.

Pour le deuxi\`eme point, on dispose du diagramme commutatif
\[\xymatrix{
\br Y \ar[r] \ar@{^{(}->}[d] & \br X \ar@{^{(}->}[rd] \ar[d] \\
\br(k(Y)) \ar[r]^{\sim} & \br X_{\eta} \ar@{^{(}->}[r] & \br (k(X)),
}\]
o\`u les fl\`eches $\hookrightarrow$ d\'esignent des injections (cf. \cite[Thm. 5.11]{ColliotSansucChili}) et celle d'en bas \`a gauche est un isomorphisme d'apr\`es
le lemme \ref{lemme Colliot}. Il est facile de voir alors que pour $\alpha\in \br X$ il existe un unique $\beta\in\br(k(Y))$ ayant la m\^eme image que $\alpha$ dans
$\br (k(X))$. Il suffit alors de montrer que $\beta\in\br Y$ pour conclure. Or, d'apr\`es \cite[\S 4]{ColliotCIME} et le fait que $X_y$ admet une composante
connexe g\'eom\'etriquement int\`egre de multiplicit\'e 1, le r\'esidu de $\beta$ au point $y$ (i.e. son image par l'application
$\delta_{A_y}:\br k(Y)\to H^1(k(y),\q/\z)$) est trivial. Ceci \'etant vrai pour tout point $y$ de $Y$ de codimension 1, les th\'eor\`emes de puret\'e de Grothendieck
(cf. \cite[Thm. 4.1.1, Prop. 4.2.3]{ColliotSantaBarbara}) nous disent que $\beta$ provient bien de $\br Y$, ce qui conclut.
\end{proof}

Enfin, pour conclure cet appendice, on passe \`a la d\'emonstration de la proposition \ref{proposition remarque Colliot}.

\begin{proof}[D\'emonstration de la proposition \ref{proposition remarque Colliot}]
Quitte \`a plonger $G$ dans un troisi\`eme groupe $G'_3$ semi-simple simplement connexe et \`a comparer respectivement $V_1$ et $V_2$ avec $V_3:=G\backslash G'_3$,
on peut se restreindre au cas o\`u l'un des $r_i$, disons $r_1$, est \`a noyau trivial. Soit alors $H$ le noyau de $r_2$. On consid\`ere le plongement
$r:G\hookrightarrow G'_1\times G'_2$ induit par $r_1$ et $r_2$. La $k$-vari\'et\'e $V=G\backslash (G'_1\times G'_2)$ se projette sur $V_1=r_1(G)\backslash G'_1$
et sur $V_2=r_2(G)\backslash G'_2$. La premi\`ere projection fait de $V$ un espace principal homog\`ene sous $G'_2$ sur $V_1$, tandis que la deuxi\`eme fait
de $V$ un espace homog\`ene sous $G'_1$ sur $V_2$ \`a stabilisateur $r_1(H)$, donc simplement connexe. Il existe des $k$-compactifications lisses $X$, $X_1$
et $X_2$ de $V$, $V_1$ et $V_2$ respectivement et des morphismes $X\to X_1$ et $X\to X_2$ \'etendant $V\to V_1$ et $V\to V_2$ (cf. par exemple
\cite[\S 1.2.2]{BorovoiKunyavskii}). Il suffit alors d'appliquer deux fois la proposition \ref{proposition Colliot} pour trouver des isomorphismes
$\br X_1\xrightarrow{\sim}\br X$ et $\br X_2\xrightarrow{\sim}\br X$, ce qui conclut.
\end{proof}

\bibliographystyle{alpha}
\bibliography{Brnral}

\begin{thebibliography}{SGA70b}

\bibitem[BDH13]{BDH}
Mikhail Borovoi, Cyril Demarche, and David Harari.
\newblock Complexes de groupes de type multiplicatif et groupe de {B}rauer non
  ramifi\'e des espaces homog\`enes.
\newblock {\em Ann. Sci. Ec. Norm. Sup.}, 46:651--692, 2013.

\bibitem[BK85]{Bogomolov_Katsylo}
F.~A. Bogomolov and P.~I. Katsylo.
\newblock Rationality of some quotient varieties.
\newblock {\em Mat. Sb. (N.S.)}, 126(168)(4):584--589, 1985.

\bibitem[BK00]{BorovoiKunyavskii}
Mikhail Borovoi and Boris Kunyavski{\u\i}.
\newblock Formulas for the unramified {B}rauer group of a principal homogeneous
  space of a linear algebraic group.
\newblock {\em J. Algebra}, 225(2):804--821, 2000.

\bibitem[CGP10]{PseudoRedGps}
Brian Conrad, Ofer Gabber, and Gopal Prasad.
\newblock {\em Pseudo-reductive groups}, volume~17 of {\em New Mathematical
  Monographs}.
\newblock Cambridge University Press, Cambridge, 2010.

\bibitem[CGR06]{Gille_Ch_R_2}
V.~Chernousov, P.~Gille, and Z.~Reichstein.
\newblock Resolving {$G$}-torsors by abelian base extensions.
\newblock {\em J. Algebra}, 296(2):561--581, 2006.

\bibitem[CGR08]{Gille_Ch_R}
V.~Chernousov, P.~Gille, and Z.~Reichstein.
\newblock Reduction of structure for torsors over semilocal rings.
\newblock {\em Manuscripta Math.}, 126(4):465--480, 2008.

\bibitem[CT95]{ColliotSantaBarbara}
J.-L. Colliot-Th{\'e}l{\`e}ne.
\newblock Birational invariants, purity and the {G}ersten conjecture.
\newblock In {\em {$K$}-theory and algebraic geometry: connections with
  quadratic forms and division algebras ({S}anta {B}arbara, {CA}, 1992)},
  volume~58 of {\em Proc. Sympos. Pure Math.}, pages 1--64. Amer. Math. Soc.,
  Providence, RI, 1995.

\bibitem[CT11]{ColliotCIME}
Jean-Louis Colliot-Th{\'e}l{\`e}ne.
\newblock Vari\'et\'es presque rationnelles, leurs points rationnels et leurs
  d\'eg\'en\'erescences.
\newblock In {\em Arithmetic geometry}, volume 2009 of {\em Lecture Notes in
  Math.}, pages 1--44. Springer, Berlin, 2011.

\bibitem[CT12]{ColliotBrnr}
J.-L. Colliot-Th{\'e}l{\`e}ne.
\newblock Une remarque sur le groupe de {B}rauer non ramifi\'e d'un quotient
  {$G/H$} avec {$G$} semi-simple simplement connexe.
\newblock Preprint :
  http://www.math.u-psud.fr/$\sim$colliot/GH25septembre12.pdf, 2012.

\bibitem[CTK98]{ColliotKunyavskii}
J.-L. Colliot-Th{\'e}l{\`e}ne and B.~{\`E}. Kunyavski{\u\i}.
\newblock Groupe de {B}rauer non ramifi\'e des espaces principaux homog\`enes
  de groupes lin\'eaires.
\newblock {\em J. Ramanujan Math. Soc.}, 13(1):37--49, 1998.

\bibitem[CTK06]{ColliotKunyavskii2}
Jean-Louis Colliot-Th{\'e}l{\`e}ne and Boris~{\`E}. Kunyavski{\u\i}.
\newblock Groupe de {P}icard et groupe de {B}rauer des compactifications lisses
  d'espaces homog\`enes.
\newblock {\em J. Algebraic Geom.}, 15(4):733--752, 2006.

\bibitem[CTS07]{ColliotSansucChili}
Jean-Louis Colliot-Th{\'e}l{\`e}ne and Jean-Jacques Sansuc.
\newblock The rationality problem for fields of invariants under linear
  algebraic groups (with special regards to the {B}rauer group).
\newblock In {\em Algebraic groups and homogeneous spaces}, Tata Inst. Fund.
  Res. Stud. Math., pages 113--186. Tata Inst. Fund. Res., Mumbai, 2007.

\bibitem[FSS98]{FSS}
Yuval~Z. Flicker, Claus Scheiderer, and R.~Sujatha.
\newblock Grothendieck's theorem on non-abelian {$H^2$} and local-global
  principles.
\newblock {\em J. Amer. Math. Soc.}, 11(3):731--750, 1998.

\bibitem[Gal92]{Galitskii}
L.~Yu. Galitski{\u\i}.
\newblock On the existence of {G}alois sections.
\newblock In {\em Lie groups, their discrete subgroups, and invariant theory},
  volume~8 of {\em Adv. Soviet Math.}, pages 65--68. Amer. Math. Soc.,
  Providence, RI, 1992.

\bibitem[Gro62]{GrothendieckPicard}
Alexander Grothendieck.
\newblock Technique de descente et th\'eor\`emes d'existence en g\'eom\'etrie
  alg\'ebrique. {VI}. {L}es sch\'emas de {P}icard: propri\'et\'es
  g\'en\'erales.
\newblock In {\em Fondements de la g\'eom\'etrie alg\'ebrique. [{E}xtraits du
  {S}\'eminaire {B}ourbaki, 1957--1962.]}, page Exp.\ No.\ 236. Secr\'etariat
  math\'ematique, Paris, 1962.

\bibitem[Kat83]{Katsylo}
P.~I. Katsylo.
\newblock Rationality of orbit spaces of irreducible representations of the
  group {${\rm SL}_{2}$}.
\newblock {\em Izv. Akad. Nauk SSSR Ser. Mat.}, 47(1):26--36, 1983.

\bibitem[LA14]{GLABrnral}
Giancarlo Lucchini~Arteche.
\newblock Groupe de {B}rauer non ramifi\'e des espaces homog\`enes \`a
  stabilisateur fini.
\newblock {\em J. Algebra}, 411:129--181, 2014.

\bibitem[NSW08]{NSW}
J{\"u}rgen Neukirch, Alexander Schmidt, and Kay Wingberg.
\newblock {\em Cohomology of number fields}, volume 323 of {\em Grundlehren der
  Mathematischen Wissenschaften}.
\newblock Springer-Verlag, Berlin, second edition, 2008.

\bibitem[Pop94]{Popov}
Vladimir Popov.
\newblock Sections in invariant theory.
\newblock In {\em The {S}ophus {L}ie {M}emorial {C}onference ({O}slo, 1992)},
  pages 315--361. Scand. Univ. Press, Oslo, 1994.

\bibitem[San81]{Sansuc81}
J.-J. Sansuc.
\newblock Groupe de {B}rauer et arithm\'etique des groupes alg\'ebriques
  lin\'eaires sur un corps de nombres.
\newblock {\em J. Reine Angew. Math.}, 327:12--80, 1981.

\bibitem[Ser65]{SerreZeta}
Jean-Pierre Serre.
\newblock Zeta and {$L$} functions.
\newblock In {\em Arithmetical {A}lgebraic {G}eometry ({P}roc. {C}onf. {P}urdue
  {U}niv., 1963)}, pages 82--92. Harper \& Row, New York, 1965.

\bibitem[Ser02]{SerreCohGal}
Jean-Pierre Serre.
\newblock {\em Galois cohomology}.
\newblock Springer Monographs in Mathematics. Springer-Verlag, Berlin,
  {E}nglish edition, 2002.
\newblock Translated from the French by Patrick Ion and revised by the author.

\bibitem[SGA70a]{SGA3I}
{\em Sch\'emas en groupes. {I}: {P}ropri\'et\'es g\'en\'erales des sch\'emas en
  groupes}.
\newblock S\'eminaire de G\'eom\'etrie Alg\'ebrique du Bois Marie 1962/64 (SGA
  3). Dirig\'e par M. Demazure et A. Grothendieck. Lecture Notes in
  Mathematics, Vol. 151. Springer-Verlag, Berlin, 1970.

\bibitem[SGA70b]{SGA3II}
{\em Sch\'emas en groupes. {II}: {G}roupes de type multiplicatif, et structure
  des sch\'emas en groupes g\'en\'eraux}.
\newblock S\'eminaire de G\'eom\'etrie Alg\'ebrique du Bois Marie 1962/64 (SGA
  3). Dirig\'e par M. Demazure et A. Grothendieck. Lecture Notes in
  Mathematics, Vol. 153. Springer-Verlag, Berlin, 1970.

\bibitem[SGA70c]{SGA3III}
{\em Sch\'emas en groupes. {III}: {S}tructure des sch\'emas en groupes
  r\'eductifs}.
\newblock S\'eminaire de G\'eom\'etrie Alg\'ebrique du Bois Marie 1962/64 (SGA
  3). Dirig\'e par M. Demazure et A. Grothendieck. Lecture Notes in
  Mathematics, Vol. 153. Springer-Verlag, Berlin, 1970.

\bibitem[Sko01]{Skor}
Alexei Skorobogatov.
\newblock {\em Torsors and rational points}, volume 144 of {\em Cambridge
  Tracts in Mathematics}.
\newblock Cambridge University Press, Cambridge, 2001.

\bibitem[Spr66]{SpringerH2}
T.~A. Springer.
\newblock Nonabelian {$H^{2}$} in {G}alois cohomology.
\newblock In {\em Algebraic {G}roups and {D}iscontinuous {S}ubgroups ({P}roc.
  {S}ympos. {P}ure {M}ath., {B}oulder, {C}olo., 1965)}, pages 164--182. Amer.
  Math. Soc., Providence, R.I., 1966.

\end{thebibliography}

\end{document}